\documentclass[a4paper, 11pt,reqno]{amsart}
\usepackage{amsthm,amssymb,amsmath,graphicx,psfrag,enumitem,mathrsfs}
\usepackage{mathtools}
\usepackage[foot]{amsaddr}
\usepackage[british]{babel}
\usepackage[babel]{microtype}
\usepackage[margin=1in]{geometry}

\usepackage[textsize=footnotesize]{todonotes}

\usepackage{algorithm}
\usepackage{algpseudocode}
\usepackage{tikz}

\usetikzlibrary{backgrounds}

\usepackage[numbers]{natbib}

\makeatletter
\def\NAT@spacechar{~}
\makeatother

\usepackage{hyperref}
\usepackage[capitalise]{cleveref}
\hypersetup{colorlinks=true,
   citecolor=blue,
   filecolor=blue,
   linkcolor=blue,
   urlcolor=blue
}

\crefname{figure}{Figure}{Figures}
\Crefname{figure}{Figure}{Figures}

\allowdisplaybreaks

\newtheorem{definition}{Definition}[section]

\newtheorem{proposition}[definition]{Proposition}
\newtheorem{theorem}[definition]{Theorem}
\newtheorem{corollary}[definition]{Corollary}
\newtheorem{lemma}[definition]{Lemma}

\newtheorem{remark}[definition]{Remark}

\numberwithin{equation}{section}

\newcommand{\bigO}{\ensuremath{O}}

\newcommand{\comment}[1]{}

\newcommand{\dist}{{\rm dist}}

\renewcommand{\epsilon}{\varepsilon}

\renewcommand{\deg}{\operatorname{deg}}

\newcommand{\COMMENT}[1]{}

\newcommand{\urlprefix}{}

\title[Edge correlations in random regular hypergraphs and subgraph testing]{Edge correlations in random regular hypergraphs and applications to subgraph testing}

\address{School of Mathematics, University of Birmingham, 
Edgbaston, Birmingham, B15 2TT, United Kingdom}
\email{axe673@bham.ac.uk, f.joos@bham.ac.uk, d.kuhn@bham.ac.uk, d.osthus@bham.ac.uk}

\author{Alberto Espuny D\'iaz}
\author{Felix Joos}
\author{Daniela K\"uhn}
\author{Deryk Osthus}

\thanks{
The research leading to these results was partially supported by the Deutsche Forschungsgemeinschaft (DFG, German Research Foundation) -- 339933727 (F.~Joos); 
the EPSRC, grant no.~EP/N019504/1 (D.~K\"uhn);
the Royal Society and the Wolfson Foundation (D.~K\"uhn),
and the European Research Council under the European Union's Seventh Framework Programme (FP/2007--2013) / ERC Grant 306349 (D.~Osthus).}

\date{\today}

\begin{document}

\begin{abstract}
Compared to the classical binomial random (hyper)graph model, the study of random regular hypergraphs is made more challenging due to correlations between the occurrence of different edges.
We develop an edge-switching technique for hypergraphs which allows us to show that these correlations are limited for a large range of densities.
This extends some previous results of Kim, Sudakov and Vu for graphs.
From our results we deduce several corollaries on subgraph counts in random $d$-regular hypergraphs.
We also prove a conjecture of Dudek, Frieze, Ruci{\' n}ski and {\v S}ileikis on the threshold for the existence of an $\ell$-overlapping Hamilton cycle in a random $d$-regular $r$-graph.

Moreover, we apply our results to prove bounds on the query complexity of testing subgraph-freeness.
The problem of testing subgraph-freeness in the general graphs model was first studied by Alon, Kaufman, Krivelevich and Ron, who obtained several bounds on the query complexity of testing triangle-freeness.
We extend some of these previous results beyond the triangle setting and to the hypergraph setting.
\end{abstract}
\maketitle
\thispagestyle{empty}

\section{Introduction}


\subsection{Random regular graphs}

While the consideration of random $d$-regular graphs is very natural and has a long history, this model is much more difficult to analyze than the seemingly similar $\mathcal{G}(n,p)$ and $\mathcal{G}(n,m)$ models due to the dependencies between edges (here $\mathcal{G}(n,p)$ refers to the binomial $n$-vertex random graph model with edge probability $p$ and $\mathcal{G}(n,m)$ refers to the uniform distribution on all $n$-vertex graphs with $m$ edges).
For small $d$, the configuration model (due to \citet{bollobas80}) has led to numerous results on random $d$-regular graphs.
Moreover, the switching method introduced by \citet{Switchoriginal} has led to results for a much larger range of $d$ than can be handled by the configuration model.
For example, \citet{KimSudaVu07} used such ideas to show that the classical results on distributions of small subgraphs in $\mathcal{G}(n,p)$ carry over to random regular graphs.

In this paper we develop an edge switching technique for random regular $r$-uniform hypergraphs (also called $r$-graphs).
More precisely, we show that correlations between the existence of edges in a random regular $r$-graph are small even if we condition on the (non-)existence of some further edges (see \cref{coro:switchprob}).
This allows us to generalise results of \citet{KimSudaVu07} on the appearance of fixed subgraphs in a random regular graph to the hypergraph setting (see \cref{coro:copies}). Moreover, even in the graph case, we can condition on the (non-)existence of a significantly larger edge set than in \cite{KimSudaVu07}.

A general result of \citet{DFRSsandwitch} implies that one can transfer many statements from the binomial model to the random regular hypergraph model (see \cref{teor:sandwichnm}).
This allows them to deduce (from the main result of \citet{DF13}) the following: if $2 \le \ell <r$ and $n^{\ell-1} \ll d \ll n^{r-1}$, then a random $d$-regular $r$-graph a.a.s.~contains an $\ell$-overlapping Hamilton cycle, that is, a Hamilton cycle in which consecutive edges overlap in precisely $\ell$ vertices (these cycles are defined formally in \cref{section14}).
They conjectured that the lower bound provides the correct threshold in the following sense:
\begin{equation}\label{equa:hamcycle}
\begin{minipage}[c]{0.8\textwidth}\em
if\/ $2\leq\ell<r$ and\/ $d \ll n^{\ell-1}$, then a.a.s.~a random\/ $d$-regular\/ $r$-graph contains no\/ $\ell$-overlapping Hamilton cycle.
\end{minipage}\ignorespacesafterend 
\end{equation} 
Our correlation results from \cref{section2} allow us to confirm this conjecture (see \cref{coro:overlappingcycles}).
The threshold for a loose Hamilton cycle (i.e.~a $1$-overlapping Hamilton cycle) in a random $d$-regular $r$-graph was recently determined (via the configuration model) by \citet{AGIR16}. This improved earlier bounds by \citet{FRS15}.
\citet{AGIR16} also investigated the above conjecture and proved that \eqref{equa:hamcycle} holds under the much stronger condition that $d \ll n$ if $r\geq4$ and $d \ll n^{1/2}$ if $r=3$ (we do rely on their result when $d$ is constant to establish \eqref{equa:hamcycle}).
The graph case $r=2$ where $d$ is fixed is a classical result by \citet{RW92,RW94}: if $d\geq3$ is fixed, then a.a.s.~a random $d$-regular graph has a Hamilton cycle.
This was extended to larger $d$ by \citet{CFR02}\COMMENT{Their result covers $c \le d \le n/c$, but linear $d$ is covered by lots of standard Hamiltonicity criteria.}.

In a similar way, we can transfer several classical counting results for random graphs to the regular setting.
We illustrate this for Hamilton cycles, where we extend the density range of a counting result of \citet{Kri12}: for $\log n\ll d \ll n$, a.a.s.~the number of Hamilton cycles in a random $d$-regular $n$-vertex graph  is fairly close to $n! (d/n)^n$ (see \cref{coro:hamcycleskrive}). The results by
\citet{Kri12} imply the same behaviour for $d \gg e^{(\log n)^{1/2}}$.
For constant $d$, this problem was studied by \citet{Janson2}.
Similarly, we transfer a general counting result for spanning subgraphs in 
${\mathcal G}(n,m)$ due to \citet{Riordan} to the setting of random regular graphs.


\subsection{Property testing}

The running time of any ``exact'' algorithm that checks whether a given combinatorial object has a given property must be at least linear in the size of the input. 
Property testing algorithms have the potential to give much quicker answers, although at the cost of not knowing for certain if the desired property is satisfied by the object.
A property testing algorithm is usually given oracle access to the combinatorial object, and answers whether the object satisfies the property or is ``far'' from satisfying it.

To be precise, following e.g.~\citet{GGR98}, we define testers as follows.
Given a property $\mathcal{P}$, a \emph{tester} for $\mathcal{P}$ is a (possibly randomized) algorithm that is given a distance parameter $\varepsilon$ and oracle access to a structure $S$.
If $S\in\mathcal{P}$, then the algorithm must accept with probability at least $2/3$.
If $S$ is $\varepsilon$-far from $\mathcal{P}$, then the algorithm should reject with probability at least $2/3$.
If the algorithm is allowed to make an error in both cases, we say it is a \emph{two-sided error tester}; if, on the contrary, the algorithm always gives the correct answer when $S$ has the property, we say it is a \emph{one-sided error tester}.

For graphs (and, more generally, $r$-graphs) there have been two classical models for testers: one of them is the dense model, and the other is the bounded-degree model.
In the dense model, the density of the $r$-graph is assumed to be bounded away from $0$, and we say that an $r$-graph $G$ is $\varepsilon$-far from having property $\mathcal{P}$ if at least $\varepsilon n^r$ edges have to be modified (added or deleted) to turn $G$ into a graph that satisfies $\mathcal{P}$.
Many results have been proved for the dense model.
In particular, there exists a characterization of all properties which are testable with constant query complexity (by \citet{AFNS09} in the graph case and \citet{JKKO17} in the $r$-graph case).
For the bounded-degree graphs model (which assumes that the maximum degree of the input graphs is bounded by a fixed constant), several general results have also been obtained (see for example the results of \citet{BSS10} as well as \citet{NS13}).

Here, we consider the general graphs model and its generalization to $r$-graphs. 
In the general graphs model (introduced by \citet{KKR03}), 
a graph $G$ with $m$ edges is $\varepsilon$-far from having property $\mathcal{P}$ if at least $\varepsilon m$ edges have to be modified for the graph to satisfy $\mathcal{P}$.
Furthermore, we also assume that the edges are labelled in the sense that for each vertex there is an ordering of its incident edges.
It is natural to consider the following two types of queries.
Firstly, we allow vertex-pair queries, where any algorithm may take two vertices and ask whether they are joined by an edge in the graph or not.
Secondly, we allow neighbour queries, where any algorithm may take a vertex and ask which vertex is its $i$-th neighbour. 

These notions  generalise to hypergraphs in a straightforward way.
More precisely, we will consider the following general hypergraphs model, where a hypergraph with $m$ edges is $\varepsilon$-far from having property $\mathcal{P}$ if at least $\varepsilon m$ edges must be added or deleted to ensure the resulting hypergraph satisfies $\mathcal{P}$.
As in the graph case, we will consider two types of queries:
\begin{itemize}
\item Vertex-set queries: Any algorithm may take a set of $r$ vertices and ask whether they constitute an edge in the $r$-graph or not. The answer must be either yes or no.
\item Neighbour queries: Any algorithm may take a vertex and ask for its $i$-th incident edge (according to the labelling of the edges). The answer is either a set of $r-1$ vertices or an error message if the degree of the queried vertex is smaller than $i$.
\end{itemize}

In this paper we consider the property $\mathcal{P}$ of being $F$-free for fixed $r$-graphs $F$.
In the dense setting, the theory of hypergraph regularity (as developed by \citet{RS04}, \citet{RS071,RS072,RS073} as well as \citet{GowersHReg}) implies the existence of testers with constant query complexity for this problem.

However, the problem is still wide open for general graphs and hypergraphs.
\citet{AlonAl08} studied the problem of testing triangle-freeness.
In \cref{section4}, we provide lower and upper bounds for testing $F$-freeness which apply to large classes of hypergraphs $F$.
In particular, we observe that testing $F$-freeness cannot be achieved in a constant number of queries whenever $F$ is not a weak forest and the density of the graphs $G$ to be tested is somewhat below the Tur\'an threshold for $F$ (see \cref{teor:bound2}).
Based on the results of \cref{section2,section31}, we also provide a lower bound (see \cref{teor:secondbound}) which improves on \cref{teor:bound2} for a large range of parameters and $r$-graphs.
Roughly speaking, \cref{teor:secondbound} provides better bounds than \cref{teor:bound2} if the average degree $d$ of the input $r$-graph $G$ is not too small.
On the other hand, the class of admissible $F$ is more restricted.
We also provide three upper bounds on the query complexity (see \cref{section43}).

\citet{KKR03} also studied the problem of testing bipartiteness in general graphs.
It would be interesting to obtain results for the general (hyper)graphs model covering further properties and to improve the lower and upper bounds we present for testing
$F$-freeness.


\subsection{Outline of the paper}

The remainder of the paper is organised as follows.
In \cref{section2} we develop a hypergraph generalisation of the edge-switching technique to prove a correlation result (\cref{coro:switchprob}) for the event that a given edge is present in a random $d$-regular $r$-graph even if we condition on the (non-)existence of some further edges.

\Cref{section3} builds on this to obtain subgraph count results in random $d$-regular $r$-graphs.
In particular, in \cref{section31} we consider the counting problem for small fixed graphs $F$, for which we prove a concentration result, thus also obtaining the threshold for their appearance, which generalises a result of \citet{KimSudaVu07} for graphs.
We also derive bounds on the number of edge-disjoint copies of fixed subgraphs $F$ in a random $d$-regular $r$-graph, which we use in \cref{section42}.
In \cref{section32}, we combine the results from \cref{section2} with known results for $\mathcal{G}^{(r)}(n,p)$ and $\mathcal{G}^{(r)}(n,m)$ to count the number of suitable spanning subgraphs (such as Hamilton cycles) in random $d$-regular $r$-graphs.

Finally, \cref{section4} provides lower and upper bounds on the query complexity for testing subgraph freeness for small, fixed $r$-graphs $F$.
The proof of the main lower bound relies on \cref{coro:switchprob} and the counting results derived in \cref{section31}.


\subsection{Definitions and notation}\label{section14}

Given any $n\in\mathbb{N}$, we will write $[n]\coloneqq\{1,\ldots,n\}$.
Throughout the paper, we will use the standard $\bigO$ notation to compare asymptotic behaviours of functions.
Whenever this is used, we implicitly assume that the functions are non-negative.
Given $a,b,c\in\mathbb{R}$, we will write $c=a\pm b$ if $c\in[a-b,a+b]$.

An \emph{$r$-graph} (or $r$-uniform hypergraph) $H=(V,E)$ is an ordered pair where $V$ is a set of vertices, and $E\subseteq\binom{V}{r}$ is a set of $r$-subsets of $V$, called edges.
We always assume that $r$ is a fixed integer greater than $1$.
When $r=2$, we will simply refer to these as graphs and omit the presence of $r$ in any notation.
To indicate the vertex set and the edge set of a certain $r$-graph $H$ we will use the notation $V(H)$ and $E(H)$, respectively.
We will often abuse notation and write $e\in H$ to mean $e\in E(H)$, or use $E(H)$ instead of $H$ to denote the $r$-graph.
In particular, we write $|H|$ for $|E(H)|$.
The \emph{order} of an $r$-graph $H$ is $|V(H)|$ and the \emph{size} of $H$ is $|E(H)|$.
For a fixed $r$-graph $H$, we sometimes denote its number of vertices by $v_H$, while $e_H$ will denote the number of edges.

Given a vertex $v\in V(H)$, the \emph{degree} of $v$ in $H$ is $\deg_H(v)\coloneqq|\{e\in H:v\in e\}|$.
When $H$ is clear from the context, it may be dropped from the notation.
We will use $\Delta(H)$ to denote the maximum (vertex) degree of $H$, $\delta(H)$ to denote the minimum (vertex) degree of $H$ and $d(H)$ to denote its average (vertex) degree.
We say that $H$ is \emph{$d$-regular} if $\deg_H(v)=d$ for all $v\in V(H)$.
The set of vertices lying in a common edge with $v$ is called its \emph{neighbourhood} and denoted by $N_H(v)$.

The \emph{complete} $r$-graph of order $n$ is denoted by $K_n^{(r)}$.
If its vertex set $V$ is given, we denote this by $K_V^{(r)}$.
We say that an $r$-graph $H$ is \emph{$k$-partite} if there exists a partition of $V(H)$ into $k$ sets such that every edge $e\in H$ contains at most one vertex in each of the sets.
A \emph{path} $P$ between vertices $u$ and $v$, also called a \emph{$(u,v)$-path}, is an $r$-graph whose vertices admit a labelling $u,v_1,\ldots,v_k,v$ such that any two consecutive vertices lie in an edge of $P$ and each edge consists of consecutive vertices.
An $r$-graph $H$ is \emph{connected} if there exists a path joining any two vertices in $V(H)$.
The \emph{distance} between vertices $u$ and $v$ in $H$ is defined by $\dist_H(u,u)\coloneqq0$ and $\dist_H(u,v)\coloneqq\min\{|P|:P\text{ is an }(u,v)\text{-path}\}$ whenever $u\neq v$.
If there is no such path, the distance is said to be infinite.
The distance between sets of vertices $S$ and $T$ is $\dist_H(S,T)\coloneqq\min\{\dist_H(s,t):s\in S, t\in T\}$.
The \emph{diameter} of an $r$-graph $H$ is $D(H)\coloneqq\max_{(u,v)\in V(H)^2}\dist_H(u,v)$.
An $r$-graph $C$ is a $k$-\emph{overlapping cycle} of length $\ell$ if $|C|=\ell$ and the vertices of $C$ admit a cyclic labelling such that each edge in $C$ consists of $r$ consecutive vertices and any two consecutive edges have exactly $k$ vertices in common (in the natural cyclic order induced on the edges of $C$)\COMMENT{For $k>r/2$, this means non-consecutive edges also intersect.}.
When $k=1$, we refer to $C$ as a \emph{loose cycle}.
When $k=r-1$, $C$ is called a \emph{tight cycle}.
A $k$-overlapping cycle $C$ is said to be \emph{Hamiltonian} for an $r$-graph $H$ if $E(C)\subseteq E(H)$ and $V(C)=V(H)$.
We will write $C_n^k$ for a $k$-overlapping cycle of order $n$.
We say that a connected $r$-graph $H$ is a \emph{weak tree} if $|e\cap f|\leq1$ for all $e,f\in H$ with $e\neq f$, and $H$ contains no loose cycles.
We say that an $r$-graph is a \emph{weak forest} if it is the union of vertex-disjoint weak trees.
Note that, for graphs, this is the usual definition of a forest.
Given any $r$-graph $H$, its complement is denoted as $\overline{H}$.

The Erd\H os-R\'enyi random $r$-graph, also called the binomial model, is denoted by $\mathcal{G}^{(r)}(n,p)$, for $n\in\mathbb{N}$ and $p\in[0,1]$.
An $r$-graph $G^{(r)}(n,p)$ on vertex set $V$ with $|V|=n$ chosen according to this model is obtained by including each $e\in\binom{V}{r}$ with probability $p$ independently from the other edges.
For $n\in\mathbb{N}$ and $m\in[\binom{n}{r}]\cup\{0\}$, we denote by $\mathcal{G}^{(r)}(n,m)$ the set of all $r$-graphs on $n$ vertices that have exactly $m$ edges, and denote by $G^{(r)}(n,m)$ an $r$-graph chosen uniformly at random from this set.
We denote the set of all $d$-regular $r$-graphs on vertex set $V$ with $|V|=n$ by $\mathcal{G}_{n,d}^{(r)}$, for $n\in\mathbb{N}$ and $d\in[\binom{n-1}{r-1}]\cup\{0\}$, and denote by $G_{n,d}^{(r)}$ an $r$-graph chosen uniformly at random from $\mathcal{G}_{n,d}^{(r)}$.
If $H$ and $H'$ are two $r$-graphs on vertex set $V$, we define $\mathcal{G}_{n,d,H,H'}^{(r)}$ as the set of all $r$-graphs $G\in \mathcal{G}_{n,d}^{(r)}$ such that $H\subseteq G$ and $H'\subseteq\overline{G}$\COMMENT{Note that $\mathcal{G}_{n,d,H,H'}^{(r)}$ can only be non-empty if $H$ and $H'$ are edge-disjoint, all the degrees of $H$ are at most $d$, and all the degrees of $H'$ are at most $\binom{n-1}{r-1}-d$.}.
With a slight abuse of notation, we sometimes also treat $\mathcal{G}_{n,d,H,H'}^{(r)}$ as the event that $G_{n,d}^{(r)}\in\mathcal{G}_{n,d,H,H'}^{(r)}$.
Given a sequence of events $\{\mathcal{A}_n\}_{n\geq1}$, we will say that $\mathcal{A}_n$ holds \emph{asymptotically almost surely}, and write a.a.s., if $\lim_{n\to\infty}\mathbb{P}[\mathcal{A}_n]=1$.

Throughout the paper, we will often use the following observation.

\begin{remark}\label{remark:intro}
Let $r\geq2$ be an integer, and let $d=o(n^{r-1})$ be such that $r\mid nd$.
Then, there exist $d$-regular $r$-graphs on $n$ vertices.
\end{remark}

Indeed, since $r\mid nd$, we can write $r=r_1r_2$ such that $r_1\mid n$ and $r_2\mid d$. Then an $(r-r_1)$-overlapping cycle is $r_2$-regular, and thus 
an edge-disjoint union of $d/r_2$ such cycles on the same vertex set is $d$-regular.
Since $d=o(n^{r-1})$, such a set of $d/r_2$ edge-disjoint cycles can be found iteratively (see e.g.~\cite[Theorem~2]{GPW12}).%
%

The condition that $r\mid nd$ is necessary, and throughout the paper we will always implicitly assume it to hold.


\section{Edge-correlation in random regular $r$-graphs}\label{section2}

This section is devoted to estimating the probability that any fixed $r$-set of vertices forms an edge in a random $d$-regular $r$-graph, even if we require certain edges to be (not) present.
More precisely, we obtain accurate bounds on $\mathbb{P}[e\in G_{n,d}^{(r)}\mid\mathcal{G}_{n,d,H,H'}^{(r)}]$ for a large range of $d$ as long as $H$, $H'$ are sparse (see \cref{coro:switchprob}).
This result is the core ingredient for all the results in \cref{section3} and it will be used in the proof of our lower bound on the query complexity for testing $F$-freeness, for a fixed $r$-graph $F$, in \cref{section42}.

\Cref{coro:switchprob} follows immediately from \cref{lema:switch1} (which provides the upper bound) and \cref{lema:switch2} (which provides the lower bound).
To prove \cref{lema:switch1,lema:switch2} we develop a hypergraph generalization of the method of edge-switchings, which was introduced for graphs by \citet{Switchoriginal}.
The switchings we consider in the proof of \cref{lema:switch1} are similar to those used by \citet{DFRSsandwitch}.
The switchings we use in \cref{lema:switch2} are more complex however.
Moreover, to bound the number of certain `bad' configurations, the proof of
 \cref{lema:switch2} relies on \cref{lema:switch1}.
The special case of \cref{lema:switch1,lema:switch2} when $r=2$ and $H$, $H'$ have bounded size (which is much simpler to prove) was obtained by \citet{KimSudaVu07}.

\begin{lemma}\label{lema:switch1}
Let\/ $r\geq2$ be a fixed integer.
Assume that\/ $d=o(n^{r-1})$.
Suppose\/ $H,H'\subseteq\binom{V}{r}$ are two edge-disjoint\/ $r$-graphs such that\/ $|H|=o(nd)$ and\/ $\Delta(H')=o(n^{r-1})$.
Then, for all\/ $e\in\binom{V}{r}\setminus(H\cup H')$, we have
\[\mathbb{P}\left[e\in G_{n,d}^{(r)}\mid\mathcal{G}_{n,d,H,H'}^{(r)}\right]\leq(r-1)!\frac{d}{n^{r-1}}\left(1+\bigO\left(\frac{1}{n}+\frac{d}{n^{r-1}}+\frac{|H|}{nd}+\frac{\Delta(H')}{n^{r-1}}\right)\right).\]
\end{lemma}

\begin{proof}
Write $e=\{v_1,\ldots,v_r\}$ and fix this labelling of the vertices in $e$.
Let $e_1\coloneqq e$ and let $e_2,\ldots,e_r\in\binom{V}{r}$ be pairwise disjoint and also disjoint from $e_1$.
Let $f_1,\ldots,f_r\in\binom{V}{r}$ be pairwise disjoint and such that $f_i\cap e_1=\{v_i\}$ for all $i\in[r]$.
We say that $\Lambda_e\coloneqq(e_1,\ldots,e_r)$ is an \emph{out-switching configuration} and that $\Lambda_{\overline{e}}\coloneqq(f_1,\ldots,f_r)$ is an \emph{in-switching configuration}.
If, furthermore, $|e_i\cap f_j|=1$ for all $i,j\in[r]$, we say that $\Lambda_e$ and $\Lambda_{\overline{e}}$ are \emph{related}\COMMENT{This is a symmetric relation.}.

Given $\Lambda_e=(e_1,\ldots,e_r)$, we denote the number of in-switching configurations related to $\Lambda_e$ by $\lambda_{\text{in}}=\lambda_{\text{in}}(\Lambda_e)$; we claim that
\begin{equation}\label{equa:lambdaetonote}
\lambda_{\text{in}}=(r!)^{r-1}.
\end{equation}
Indeed, for each $i\in[r]\setminus\{1\}$, write $e_i=\{v_1^i,\ldots,v_r^i\}$ and let $\pi_i\colon[r]\to[r]$ be a permutation.
For each $i\in[r]$, let $f_i\coloneqq\{v_i,v_{\pi_2(i)}^2,\ldots,v_{\pi_r(i)}^r\}$.
Then, $\Lambda_{\overline{e}}\coloneqq(f_1,\ldots,f_r)$ is related to $\Lambda_e$.
In this way, each (ordered) $(r-1)$-tuple of permutations $(\pi_2,\ldots,\pi_r)$ defines a unique in-switching configuration.
On the other hand, each $\Lambda_{\overline{e}}=(f_1,\ldots,f_r)$ related to $\Lambda_e$ gives rise to a different $(r-1)$-tuple of permutations $(\pi_2,\ldots,\pi_r)$ by setting, for each $i\in[r]\setminus\{1\}$ and $j\in[r]$, $\pi_i(j)$ to be the subscript of the vertex in $e_i\cap f_j$.
There are $(r!)^{r-1}$ such tuples of permutations, so \eqref{equa:lambdaetonote} follows.

Similarly, given $\Lambda_{\overline{e}}=(f_1,\ldots,f_r)$, we denote the number of out-switching configurations related to $\Lambda_{\overline{e}}$ by $\lambda_{\text{out}}=\lambda_{\text{out}}(\Lambda_{\overline{e}})$.
We claim that
\begin{equation}\label{equa:lambdanotetoe}
\lambda_{\text{out}}=((r-1)!)^{r}.
\end{equation}
Indeed, for each $i\in[r]$, write $f_i=\{v_i,v_2^i,\ldots,v_r^i\}$ and let $\sigma_i\colon[r]\setminus\{1\}\to[r]\setminus\{1\}$ be a permutation.
For each $i\in[r]\setminus\{1\}$, let $e_i\coloneqq\{v_{\sigma_1(i)}^1,\ldots,v_{\sigma_r(i)}^r\}$.
Then, $\Lambda_e\coloneqq(e_1,\ldots,e_r)$ is related to $\Lambda_{\overline{e}}$.
Each $r$-tuple of permutations $(\sigma_1,\ldots,\sigma_r)$ defines a unique $\Lambda_e$.
On the other hand, each $\Lambda_e=(e_1,\ldots,e_r)$ related to $\Lambda_{\overline{e}}$ gives rise to a unique $r$-tuple of permutations $(\sigma_1,\ldots,\sigma_r)$\COMMENT{by setting, for each $i\in[r]$ and $j\in[r]\setminus\{1\}$, $\sigma_i(j)$ to be the subscript of the vertex in $e_j\cap f_i$. There are $((r-1)!)^r$ such tuples of permutations.}.
Thus \eqref{equa:lambdanotetoe} holds.

Let $\Omega_1,\Omega_2\subseteq\binom{V}{r}$.
We define a function $\psi$ on the set of all $r$-graphs $G$ on $V$ by $\psi(G,\Omega_1,\Omega_2)\coloneqq(G\setminus\Omega_1)\cup\Omega_2$.
Now let $G$ be an $r$-graph on $V$.
Let $\Lambda_e$ and $\Lambda_{\overline{e}}$ be related out- and in-switching configurations, respectively, such that $\Lambda_e\subseteq G$ and $\Lambda_{\overline{e}}\subseteq\overline{G}$.
An \emph{out-switching} on $G$ from $\Lambda_e$ to $\Lambda_{\overline{e}}$ is obtained by applying the operation $\psi(G,\Lambda_e,\Lambda_{\overline{e}})$ (here $\Lambda_e$ and $\Lambda_{\overline{e}}$ are viewed as (unordered) sets of edges).
We denote this out-switching by the triple $(G,\Lambda_e,\Lambda_{\overline{e}})$.
Similarly, if $\Lambda_e$ and $\Lambda_{\overline{e}}$ are related out- and in-switching configurations, respectively, such that $\Lambda_e\subseteq\overline{G}$ and $\Lambda_{\overline{e}}\subseteq G$, an \emph{in-switching} on $G$ from $\Lambda_{\overline{e}}$ to $\Lambda_e$ is the operation $\psi(G,\Lambda_{\overline{e}},\Lambda_e)$, and is denoted by $(G,\Lambda_{\overline{e}},\Lambda_e)$.
Note that $\psi(\psi(G,\Lambda_{\overline{e}},\Lambda_e),\Lambda_e,\Lambda_{\overline{e}})=G$, that is, switchings are involutions.
Furthermore, both types of switchings preserve the vertex degrees of the $r$-graph $G$ on which they act.

Let $\mathcal{F}_e\subseteq\mathcal{G}_{n,d,H,H'}^{(r)}$ be the set of all $r$-graphs $G\in\mathcal{G}_{n,d,H,H'}^{(r)}$ such that $e\in G$, and let $\mathcal{F}_{\overline{e}}\coloneqq\mathcal{G}_{n,d,H,H'}^{(r)}\setminus\mathcal{F}_e$.
We define an auxiliary bipartite multigraph $\Gamma$ with bipartition $(\mathcal{F}_e,\mathcal{F}_{\overline{e}})$ as follows.
For each $G\in\mathcal{F}_e$, consider all possible out-switchings on $G$ whose image is in $\mathcal{G}_{n,d,H,H'}^{(r)}$ (that is, all triples $(G,\Lambda_e,\Lambda_{\overline{e}})$ such that $\Lambda_e\subseteq G\setminus H$ and $\Lambda_{\overline{e}}\subseteq\overline{G}\setminus H'$ are related) and add an edge between $G$ and $\psi(G,\Lambda_e,\Lambda_{\overline{e}})$ for each such triple $(G,\Lambda_e,\Lambda_{\overline{e}})$\COMMENT{Note that for any such triple we have that $\psi(G,\Lambda_e,\Lambda_{\overline{e}})\in\mathcal{F}_{\overline{e}}$.}.
Similarly, one could consider each $G\in\mathcal{F}_{\overline{e}}$ and every possible in-switching $(G,\Lambda_{\overline{e}},\Lambda_e)$ on $G$ with $\psi(G,\Lambda_{\overline{e}},\Lambda_e)\in\mathcal{G}_{n,d,H,H'}^{(r)}$, and add an edge between $G$ and $\psi(G,\Lambda_{\overline{e}},\Lambda_e)$\COMMENT{Note that for any such triple we have that $\psi(G,\Lambda_{\overline{e}},\Lambda_e)\in\mathcal{F}_e$.}.
Both constructions result in the same multigraph $\Gamma$\COMMENT{Note that there are multiple edges because two different $\Lambda_e^1$ and $\Lambda_e^2$ with the same edges in different order have matching $\Lambda_{\overline{e}}^1$ and $\Lambda_{\overline{e}}^2$ with the same edges in different order such that the result of the switching is the same $r$-graph for both out-switchings.}.

We will use switchings to bound $\mathbb{P}[e\in G_{n,d}^{(r)}\mid\mathcal{G}_{n,d,H,H'}^{(r)}]=|\mathcal{F}_e|/|\mathcal{G}_{n,d,H,H'}^{(r)}|$ from above in terms of $\mathbb{P}[e\notin G_{n,d}^{(r)}\mid\mathcal{G}_{n,d,H,H'}^{(r)}]$.
In order to obtain this bound, we will use a double-counting argument involving the edges of $\Gamma$. 

Assume first that $G\in\mathcal{F}_{\overline{e}}$.
Let $S_{\text{in}}(G)$ be the number of in-switchings $(G,\Lambda_{\overline{e}},\Lambda_e)$ on $G$, thus $\deg_{\Gamma}(G)\leq S_{\text{in}}(G)$\COMMENT{Here we don't have equality because the current definition of $S_{\text{in}}$ includes ``impossible'' switchings, as it does not consider $H$ and $H'$.}.
We claim that
\begin{equation}\label{equa:switch1up}
S_{\text{in}}(G)\leq ((r-1)!)^{r}d^r.
\end{equation}
Clearly, $S_{\text{in}}(G)$ is at most\COMMENT{It is not equal because one could have that one of the out-switching configurations related to an in-switching configuration contained in $G$ is not contained in $\overline{G}$.} the number of in-switching configurations $\Lambda_{\overline{e}}\subseteq G$ multiplied by $\lambda_{\text{out}}$.
As $G$ is $d$-regular and $\Lambda_{\overline{e}}$ must contain an edge incident to each $v_i\in e$, there are at most $d^r$ such in-switching configurations.
This, together with \eqref{equa:lambdanotetoe}, yields \eqref{equa:switch1up}.

Assume now that $G\in\mathcal{F}_e$.
Let $\ell\coloneqq|H|$ and\/ $k'\coloneqq\Delta(H')$, and let $\eta\coloneqq\eta(n,d,\ell,k')=\frac{1}{n}+\frac{d}{n^{r-1}}+\frac{\ell}{nd}+\frac{k'}{n^{r-1}}$.
Let $S_{\text{out}}(G)$ be the number of possible out-switchings $(G,\Lambda_e,\Lambda_{\overline{e}})$ on $G$ with $\psi(G,\Lambda_e,\Lambda_{\overline{e}})\in\mathcal{G}_{n,d,H,H'}^{(r)}$; thus, $\deg_{\Gamma}(G)=S_{\text{out}}(G)$.
We claim that
\begin{equation}\label{equa:switch1low}
S_{\text{out}}(G)\geq ((r-1)!)^{r-1}(nd)^{r-1}\left(1-\bigO\left(\eta\right)\right).
\end{equation}

In order to have $\psi(G,\Lambda_e,\Lambda_{\overline{e}})\in\mathcal{G}_{n,d,H,H'}^{(r)}$ we must have $\Lambda_e\subseteq G\setminus H$ and $\Lambda_{\overline{e}}\subseteq\overline{G}\setminus H'$.
Let $\lambda_e(G)$ be the number of out-switching configurations $\Lambda_e$ with $\Lambda_e\subseteq G\setminus H$.
We first give a lower bound on $\lambda_e(G)$.

Choose $\Lambda_e=(e_1,\ldots,e_r)$ by sequentially choosing $e_2,\ldots,e_r\in G\setminus H$ in such a way that $e_i$ is disjoint from $e_1,\ldots,e_{i-1}$, for $i\in[r]\setminus\{1\}$.
As each vertex is incident to exactly $d$ edges, the number of choices for $e_i$ is at least $(nd/r-\ell-(r-1)rd)$\COMMENT{Because $(r-1)rd$ is an upper bound on the number of edges incident to $e_1, \ldots, e_{i-1}$.}.
Thus,
\begin{equation}\label{equa:switch1eq2}
\lambda_e(G)\geq\left(\frac{nd}{r}-\ell-(r-1)rd\right)^{r-1}.
\end{equation}

We say that an out-switching configuration $\Lambda_e\subseteq G\setminus H$ is \emph{good} (for $G$) if there are $\lambda_{\text{in}}$ in-switching configurations $\Lambda_{\overline{e}}\subseteq\overline{G}\setminus H'$ related to $\Lambda_e$,
and \emph{bad} (for $G$) otherwise.
Let $\lambda_{e,\text{bad}}(G)$ denote the number of bad out-switching configurations $\Lambda_e\subseteq G\setminus H$.
We now provide an upper bound on this quantity.
An out-switching configuration $\Lambda_e\subseteq G\setminus H$ can only be bad if
\begin{enumerate}[label=(\alph*)]
\item one of the edges in some $\Lambda_{\overline{e}}$ related to $\Lambda_e$, say $g$, lies in $G$, or \label{badcond11}
\item one of the edges in some $\Lambda_{\overline{e}}$ related to $\Lambda_e$, say $h$, lies in $H'$. \label{badcond12}
\end{enumerate}
In case \ref{badcond11}, the edge $g$ has to intersect $e$, so there are at most $rd$ possible such edges $g$.
Furthermore, $g\setminus e$ must intersect every edge in $\Lambda_e\setminus\{e\}$, so each edge $g$ can make at most $(r-1)!d^{r-1}$ out-switching configurations bad.
Thus, there are at most $r!d^r$ out-switching configurations which are bad because of \ref{badcond11}.
In case \ref{badcond12}, the edge $h$ has to intersect $e$, so there are at most $rk'$ such edges.
As above, it follows that there are at most $r!k'd^{r-1}$ out-switching configurations which are bad because of \ref{badcond12}.
Overall,
\begin{equation}\label{equa:switch1eq3}
\lambda_{e,\text{bad}}(G)\leq r!d^r+r!k'd^{r-1}.
\end{equation}
By combining \eqref{equa:lambdaetonote}, \eqref{equa:switch1eq2} and \eqref{equa:switch1eq3}, we have that\COMMENT{$((r-1)!)^{r-1}(nd)^{r-1}-\bigO\left(\ell(nd)^{r-2}+n^{r-2}d^{r-1}+d^ r+k'd^{r-1}\right)=((r-1)!)^{r-1}(nd)^{r-1}-\bigO\left((nd)^{r-1}\left(\frac{\ell}{nd}+\frac{1}{n}+\frac{d}{n^{r-1}}+\frac{k'}{n^{r-1}}\right)\right)=$}
\begin{align*}
S_{\text{out}}(G)&\geq(r!)^{r-1}\left(\left(\frac{nd}{r}-\ell-(r-1)rd\right)^{r-1}-r!d^r-r!k'd^{r-1}\right)\nonumber\\
&=((r-1)!)^{r-1}(nd)^{r-1}\left(1-\bigO\left(\eta\right)\right).
\end{align*}

As \eqref{equa:switch1up} and \eqref{equa:switch1low} hold for every $G\in\mathcal{F}_{\overline{e}}$ and $G\in\mathcal{F}_e$, respectively, we can use these expressions to estimate the number $|\Gamma|$ of edges in $\Gamma$.
We conclude that
\[((r-1)!)^{r-1}(nd)^{r-1}\left(1-\bigO\left(\eta\right)\right)|\mathcal{F}_e|\leq|\Gamma|\leq ((r-1)!)^{r}d^r|\mathcal{F}_{\overline{e}}|.\]
Noting that $|\mathcal{F}_{\overline{e}}|\leq|\mathcal{G}_{n,d,H,H'}^{(r)}|$ and dividing this by $|\mathcal{G}_{n,d,H,H'}^{(r)}|$ implies that
\begin{equation*}
((r-1)!)^{r-1}(nd)^{r-1}\left(1-\bigO\left(\eta\right)\right)\cdot\mathbb{P}\left[e\in G_{n,d}^{(r)}\mid\mathcal{G}_{n,d,H,H'}^{(r)}\right]\leq ((r-1)!)^{r}d^r.
\end{equation*}
Thus, we conclude that
\begin{equation*}
\mathbb{P}\left[e\in G_{n,d}^{(r)}\mid\mathcal{G}_{n,d,H,H'}^{(r)}\right]\leq\COMMENT{$(r-1)!\frac{d^r}{(nd)^{r-1}\left[1-\bigO\left(\eta\right)\right]}=$}(r-1)!\frac{d}{n^{r-1}}\left(1+\bigO\left(\eta\right)\right).\qedhere
\end{equation*}
\end{proof}

\begin{lemma}\label{lema:switch2}
Let\/ $r\geq2$ be a fixed integer.
Suppose that\/ $d=\omega(1)$ and\/ $d=o(n^{r-1})$.
Let\/ $H,H'\subseteq\binom{V}{r}$ be two edge-disjoint\/ $r$-graphs such that\/ $\Delta(H),\Delta(H')=o(d)$.
Then, for all\/ $e\in\binom{V}{r}\setminus(H\cup H')$,
\[\mathbb{P}\left[e\in G_{n,d}^{(r)}\mid\mathcal{G}_{n,d,H,H'}^{(r)}\right]\geq(r-1)!\frac{d}{n^{r-1}}\left(1-\bigO\left(\frac{1}{n}+\frac{1}{d}+\frac{d}{n^{r-1}}+\frac{\Delta(H)}{d}+\frac{\Delta(H')}{d}\right)\right).\]
\end{lemma}

\begin{proof}
Our strategy is similar as in \cref{lema:switch1}, but we change the definition of a switching configuration.
Write $e=\{v_1,\ldots,v_r\}$.
Let $e_1,\ldots,e_r\in\binom{V}{r}$ be such that, for each $i\in[r]$, $v_i\notin e_i$ and there is a vertex $u_i\in e_i\setminus e$ such that $u_i\notin e_j$ for all $j\in[r]\setminus\{i\}$.
Let $f_1,\ldots,f_r\in\binom{V}{r}\setminus\{e\}$ be distinct such that $v_i\in f_i$, and let $f\in\binom{V}{r}$ be disjoint from $f_1,\ldots,f_r$.
We say that $\Lambda_e\coloneqq(e,e_1,\ldots,e_r)$ is an \emph{out-switching configuration} and that $\Lambda_{\overline{e}}\coloneqq(f_1,\ldots,f_r,f)$ is an \emph{in-switching configuration}.
We say that $\Lambda_e$ and $\Lambda_{\overline{e}}$ are \emph{related} if, for each $i\in[r]$, one can find a set $A_i\in\binom{V}{r-1}$ such that $e_i\cap f_i=A_i$, and $f=(e_1\setminus A_1)\cup\ldots\cup(e_r\setminus A_r)$\COMMENT{The sets $A_i$ are not necessarily disjoint.} (note that in this case we must have $A_i=f_i\setminus\{v_i\}$).
See \cref{figu:switch3} for an illustration.
Given related out- and in-switching configurations $\Lambda_e=(e,e_1,\ldots,e_r)$ and $\Lambda_{\overline{e}}=(f_1,\ldots,f_r,f)$, we will always write $A_i\coloneqq e_i\cap f_i$ and $\{u_i\}\coloneqq e_i\setminus f_i$ for $i\in[r]$.
It is easy to check that this definition of $u_i$ implies that $\{u_i\}=e_i\cap f$\COMMENT{If $u_i\coloneqq e_i\setminus f_i$ then clearly $u_i\in e_i\cap f$. Conversely, if $|e_i\cap f|\geq2$ then $u_j\in e_i\cap f$ for some $u_j=e_j\setminus f_j$ with $j\neq i$. But then $u_j\in A_i\subseteq f_i$, a contradiction to $f_i\cap f=\varnothing$. In particular, this argument shows that $u_j\notin e_i$ for all $i\neq j$. Finally, $u_i\in f$ and so $u_i\notin f_1\cup\ldots\cup f_r$, thus $u_i\notin e$. Thus the $u_i$ defined as $u_i\coloneqq e_i\setminus f_i$ are indeed as in the definition of an out-switching configuration.} and $u_i\notin e_j$ for all $j\in[r]\setminus\{i\}$.
So $u_i$ is indeed as required in the definition of an out-switching configuration.

Given $\Lambda_e=(e,e_1,\ldots,e_r)$, we denote the number of in-switching configurations related to $\Lambda_e$ by $\lambda_{\text{in}}(\Lambda_e)$.
We claim that
\begin{equation}\label{equa:swithc2eq1}
\lambda_{\text{in}}(\Lambda_e)\leq r^r.
\end{equation}
Indeed, in order to obtain an in-switching configuration $\Lambda_{\overline{e}}=(f_1,\ldots,f_r,f)$ related to $\Lambda_e$ one has to choose $u_i\in e_i$ for each $i\in[r]$. There are at most $r$ choices for each $u_i$.
Each (admissible) choice of $u_i$ uniquely determines $f_i$, and thus they determine $f$.

Similarly, given $\Lambda_{\overline{e}}=(f_1,\ldots,f_r,f)$, we denote the number of out-switching configurations related to $\Lambda_{\overline{e}}$ by $\lambda_{\text{out}}=\lambda_{\text{out}}(\Lambda_{\overline{e}})$.
We claim that
\begin{equation}\label{equa:swithc2eq2}
\lambda_{\text{out}}=r!.
\end{equation}
This holds because, for each $i\in[r]$, the edge $e_i$ must contain $f_i\setminus\{v_i\}=A_i$ and one vertex $u_i\in f$, hence each permutation of the labels of the vertices in $f$ results in a different $\Lambda_e$.

\begin{figure}
\begin{tikzpicture}[scale=0.7]
\draw [blue, thick, fill=blue, fill opacity=0.25] (0,0) circle (4.44264068712cm);
\draw [blue, thick, fill=white, fill opacity=1] (0,0) circle (4.04264068712cm);
\draw (1,1) node[circle,fill,inner sep=1pt]{};
\draw (-1,1) node[circle,fill,inner sep=1pt]{};
\draw (1,-1) node[circle,fill,inner sep=1pt]{};
\draw (-1,-1) node[circle,fill,inner sep=1pt]{};
\draw [red, thick, rounded corners, fill=red, fill opacity=0.05] (1.2,1.2) -- (1.2,-1.2) -- (-1.2,-1.2) -- (-1.2,1.2) -- cycle;
\node at (0,0) {$e$};
\node at (0.7,0.7) {$v_1$};
\node at (0.7,-0.7) {$v_4$};
\node at (-0.7,-0.7) {$v_3$};
\node at (-0.7,0.7) {$v_2$};
\node at (1.5,1.5) {$f_1$};
\node at (1.25,-1.5) {$f_4$};
\node at (-1.25,-1.5) {$f_3$};
\node at (-1.5,1.5) {$f_2$};
\node at (2.5,2.5) {$e_1$};
\node at (2.5,-2.5) {$e_4$};
\node at (-2.5,-2.5) {$e_3$};
\node at (-2.5,2.5) {$e_2$};
\node at (3.35,3.35) {$u_1$};
\node at (3.35,-3.35) {$u_4$};
\node at (-3.35,-3.35) {$u_3$};
\node at (-3.35,3.35) {$u_2$};
\node at (1.5,2.5) {$A_1$};
\node at (2.5,-1.35) {$A_4$};
\node at (-2.5,-1.35) {$A_3$};
\node at (-1.5,2.5) {$A_2$};
\node at (4.25,0) {$f$};
\draw (3,1) node[circle,fill,inner sep=1pt]{};
\draw (1,3) node[circle,fill,inner sep=1pt]{};
\draw (2,2) node[circle,fill,inner sep=1pt]{};
\draw [blue, thick, fill=blue, fill opacity=0.25] (0.8,1) arc [radius=0.2 cm, start angle=180, end angle=270] -- (3,0.8) arc [radius=0.2 cm, start angle=-90, end angle=45] -- (1.14142135623,3.14142135623) arc [radius=0.2 cm, start angle=45, end angle=180] -- cycle;
\draw (-1,3) node[circle,fill,inner sep=1pt]{};
\draw (-3,1) node[circle,fill,inner sep=1pt]{};
\draw (-2,2) node[circle,fill,inner sep=1pt]{};
\draw [blue, thick, fill=blue, fill opacity=0.25] (-0.8,1) arc [radius=0.2 cm, start angle=0, end angle=-90] -- (-3,0.8) arc [radius=0.2 cm, start angle=270, end angle=135] -- (-1.14142135623,3.14142135623) arc [radius=0.2 cm, start angle=135, end angle=0] -- cycle;
\draw (-3,-1) node[circle,fill,inner sep=1pt]{};
\draw (-1.5,-2) node[circle,fill,inner sep=1pt]{};
\draw (0,-3) node[circle,fill,inner sep=1pt]{};
\draw (1.5,-2) node[circle,fill,inner sep=1pt]{};
\draw (3,-1) node[circle,fill,inner sep=1pt]{};
\draw [blue, thick, fill=blue, fill opacity=0.25] (-0.8211145618,-0.9105572809) arc [radius=0.2 cm, start angle=26.5650512, end angle=90] -- (-3,-0.8) arc [radius=0.2 cm, start angle=90, end angle=236.3099325] -- (-0.11094003924,-3.16641005886) arc [radius=0.2 cm, start angle=236.3099325, end angle=386.5650512] -- cycle;
\draw [blue, thick, fill=blue, fill opacity=0.25] (0.8211145618,-0.9105572809) arc [radius=0.2 cm, start angle=153.4349488, end angle=90] -- (3,-0.8) arc [radius=0.2 cm, start angle=90, end angle=-56.3099325] -- (0.11094003924,-3.16641005886) arc [radius=0.2 cm, start angle=-56.3099325, end angle=-206.5650512] -- cycle;
\draw (3,3) node[circle,fill,inner sep=1pt]{};
\draw (-3,3) node[circle,fill,inner sep=1pt]{};
\draw (3,-3) node[circle,fill,inner sep=1pt]{};
\draw (-3,-3) node[circle,fill,inner sep=1pt]{};
\draw [red, thick, fill=red, fill opacity=0.05] (3.2,3) arc [radius=0.2 cm, start angle=0, end angle=90] -- (1,3.2) arc [radius=0.2 cm, start angle=90, end angle=225] -- (2.85857864377,0.85857864377) arc [radius=0.2 cm, start angle=225, end angle=360] -- cycle;
\draw [red, thick, fill=red, fill opacity=0.05] (-3.2,3) arc [radius=0.2 cm, start angle=180, end angle=90] -- (-1,3.2) arc [radius=0.2 cm, start angle=90, end angle=-45] -- (-2.85857864377,0.85857864377) arc [radius=0.2 cm, start angle=-45, end angle=-180] -- cycle;
\draw [red, thick, fill=red, fill opacity=0.05] (3.2,-3) arc [radius=0.2 cm, start angle=0, end angle=-90] -- (0,-3.2) arc [radius=0.2 cm, start angle=-90, end angle=-236.30993245] -- (2.88905996076,-0.83358994114) arc [radius=0.2 cm, start angle=-236.30993245, end angle=-360] -- cycle;
\draw [red, thick, fill=red, fill opacity=0.05] (-3.2,-3) arc [radius=0.2 cm, start angle=180, end angle=270] -- (0,-3.2) arc [radius=0.2 cm, start angle=270, end angle=416.3099325] -- (-2.88905996076,-0.83358994114) arc [radius=0.2 cm, start angle=416.3099325, end angle=540] -- cycle;
\end{tikzpicture}
\caption{Representation of a switching for \cref{lema:switch2} in the case $r=4$. Shaded (blue) edges represent an in-switching configuration, while clear (red) ones represent an out-switching configuration.}\label{figu:switch3}
\end{figure}
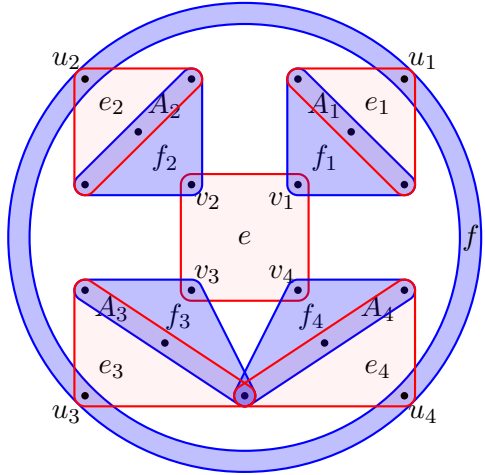

We define $\psi(G,\Lambda_e,\Lambda_{\overline{e}})$, $\mathcal{F}_e$, $\mathcal{F}_{\overline{e}}$ and $\Gamma$ as in the proof of \cref{lema:switch1}.
As before, neither out- nor in-switchings on an $r$-graph $G$ change the vertex degrees.

Assume first that $G\in\mathcal{F}_e$.
Let $S_{\text{out}}(G)$ be the number of possible out-switchings $(G,\Lambda_e,\Lambda_{\overline{e}})$ on $G$ satisfying that $\psi(G,\Lambda_e,\Lambda_{\overline{e}})\in\mathcal{G}_{n,d,H,H'}^{(r)}$.
Thus $\deg_{\Gamma}(G)=S_{\text{out}}(G)$.
Let $S_{\text{out}}\coloneqq\sum_{G\in\mathcal{F}_e}S_{\text{out}}(G)$ be the number of edges incident to $\mathcal{F}_e$ in $\Gamma$.
We claim that 
\begin{equation}\label{equa:switch2up1}
S_{\text{out}}(G)\leq(nd)^r.
\end{equation}
Indeed, \eqref{equa:swithc2eq1} implies that $S_{\text{out}}(G)$ is at most the number of out-switching configurations $\Lambda_e\subseteq G$ multiplied by $r^r$.
The number of such out-switching configurations is given by the choice of $(e_1,\ldots,e_r)$, so there are at most $(nd/r)^r$ such configurations.
This yields \eqref{equa:switch2up1}.
As this is true for every $G$,
\begin{equation}\label{equa:switch2up}
S_{\text{out}}\leq|\mathcal{F}_e|(nd)^r.
\end{equation}

Consider now any $r$-graph $G\in\mathcal{F}_{\overline{e}}$.
Let $S_{\text{in}}(G)$ be the number of possible in-switchings $(G,\Lambda_{\overline{e}},\Lambda_e)$ on $G$ satisfying that $\psi(G,\Lambda_{\overline{e}},\Lambda_e)\in\mathcal{G}_{n,d,H,H'}^{(r)}$.
Thus $\deg_{\Gamma}(G)=S_{\text{in}}(G)$.
Let $S_{\text{in}}\coloneqq\sum_{G\in\mathcal{F}_{\overline{e}}}S_{\text{in}}(G)$ be the number of edges incident to $\mathcal{F}_{\overline{e}}$ in $\Gamma$.
Let $T_{\text{in}}(G)$ denote the number of in-switching configurations $\Lambda_{\overline{e}}\subseteq G$.
As an in-switching configuration is given by $r$ edges, one incident to each of the vertices of $e$, and one more edge which is disjoint from the previous ones, by choosing each edge in turn and taking into consideration that $G$ is $d$-regular, we conclude that
\begin{equation}\label{equa:switch2up2}
T_{\text{in}}(G)\leq \frac{nd^{r+1}}{r}.
\end{equation}
For a lower bound on $T_{\text{in}}(G)$, observe that there are exactly $d$ choices for $f_1$\COMMENT{This is because, by assumption, $G\in\mathcal{F}_{\overline{e}}$ so $e\notin G$.}.
Then, $f_2$ can be chosen in at least $d-1$ ways\COMMENT{, considering that the chosen $f_1$ might also contain $v_2$ and $f_2$ must be distinct from $f_1$}.
More generally, there are at least $(d-r)^r$ choices for $(f_1,\ldots,f_r)$.
Finally, $f$ must be chosen disjoint from $f_1,\ldots,f_r$, so there are at least $nd/r-r^2d$ choices.
Overall,
\begin{equation}\label{equa:switch2low2}
T_{\text{in}}(G)\geq(d-r)^r\left(\frac{nd}{r}-r^2d\right)=\frac{nd^{r+1}}{r}\left(1-\bigO\left(\frac{1}{d}+\frac{1}{n}\right)\right).
\end{equation}

We say that an in-switching configuration $\Lambda_{\overline{e}}\subseteq G$ is \emph{good} (for $G$) if there are $\lambda_{\text{out}}$ out-switching configurations $\Lambda_e\subseteq\overline{G}$ related to $\Lambda_{\overline{e}}$ which satisfy $\psi(G,\Lambda_{\overline{e}},\Lambda_e)\in\mathcal{G}_{n,d,H,H'}^{(r)}$.
We say that $\Lambda_{\overline{e}}$ is \emph{bad} (for $G$) otherwise.
An in-switching configuration $\Lambda_{\overline{e}}=(f_1,\ldots,f_r,f)$ is bad for $G$ if and only if any of the following occur:
\begin{enumerate}[label=(\alph*)]
\item $(f_i\setminus\{v_i\})\cup\{v\}\in H$ for some $i\in[r]$ and $v\in f$. \label{badcond1}
\item $(f_i\setminus\{v_i\})\cup\{v\}\in H'$ for some $i\in[r]$ and $v\in f$.\label{badcond3}
\item $f_i\in H$ for some $i\in[r]$ or $f\in H$. \label{badcond2}
\item Neither \ref{badcond1} nor \ref{badcond3} hold, but $(f_i\setminus\{v_i\})\cup\{v\}\in G$ for some $i\in[r]$ and $v\in f$. \label{badcond15}
\end{enumerate}

For each $G\in\mathcal{F}_{\overline{e}}$, let $\mathcal{L}(G)$ denote the set of in-switching configurations $\Lambda_{\overline{e}}$ with $\Lambda_{\overline{e}}\subseteq G$.
Consider the set $\varOmega\coloneqq\{(G,\Lambda_{\overline{e}})\mid G\in\mathcal{F}_{\overline{e}},\Lambda_{\overline{e}}\in\mathcal{L}(G)\}$\COMMENT{
Note that the subsets of $\varOmega$ defined by each $\Lambda_{\overline{e}}$ define a partition of $\varOmega$}.
We say that a pair $(G,\Lambda_{\overline{e}})$ is \emph{bad} if $\Lambda_{\overline{e}}$ is bad for $G$.

Let $k\coloneqq\Delta(H)$, $k'\coloneqq\Delta(H')$.
We first count the number of in-switching configurations in $\mathcal{L}(G)$ which are bad because of \ref{badcond1}--\ref{badcond2}.
For this, fix an $r$-graph $G\in\mathcal{F}_{\overline{e}}$.
Let $T_\text{a}(G)$ be the number of in-switching configurations which are bad because of \ref{badcond1}.
Fix $e^*\in H$ and $i\in[r]$.
To count the number of in-switching configurations $\Lambda_{\overline{e}}=(f_1,\ldots,f_r,f)\in\mathcal{L}(G)$ with $(f_i\setminus\{v_i\})\cup\{v\}=e^*$ for some $v\in f$, note that there are at most $r$ choices for $v$, and then at most $d$ choices for $f$ (since $v\in f$).
Then we must have $f_i=(e^*\setminus\{v\})\cup\{v_i\}$.
Finally, there are at most $d$ choices for each $f_j$ with $j\in[r]\setminus\{i\}$ (since $v_j\in f_j$).
Therefore, $T_\text{a}(G)\leq |H|\cdot r\cdot r\cdot d\cdot d^{r-1}\leq rnkd^r$.
Let $T_\text{a}\coloneqq\sum_{G\in\mathcal{F}_{\overline{e}}}T_\text{a}(G)$ be the number of pairs $(G,\Lambda_{\overline{e}})$ which are bad because of \ref{badcond1}.
Then,
\begin{equation}\label{equa:abad}
T_\text{a}\leq |\mathcal{F}_{\overline{e}}|rnkd^r.
\end{equation}

Similarly, for $G\in\mathcal{F}_{\overline{e}}$, let $T_\text{b}(G)$ be the number of in-switching configurations which are bad because of \ref{badcond3}.
As above, one can show that $T_\text{b}(G)\leq |H'|\cdot r\cdot r\cdot d\cdot d^{r-1}\leq rnk'd^r$.
Let $T_\text{b}\coloneqq\sum_{G\in\mathcal{F}_{\overline{e}}}T_\text{b}(G)$ be the number of pairs $(G,\Lambda_{\overline{e}})$ which are bad because of \ref{badcond3}.
Then,
\begin{equation}\label{equa:cbad}
T_\text{b}\leq |\mathcal{F}_{\overline{e}}|rnk'd^r.
\end{equation}

Next, for $G\in\mathcal{F}_{\overline{e}}$, let $T_\text{c}(G)$ be the number of in-switching configurations which are bad because of \ref{badcond2}. 
Given $i\in[r]$, there are at most $k$ choices for $f_i\in H$ (as $v_i\in f_i$), and the remaining edges in the in-switching configuration can be chosen in at most $d^{r-1}nd/r$ ways.
Similarly, if $f\in H$, then the remaining edges in the in-switching configuration can be chosen in at most $d^r$ ways.
Therefore, $T_\text{c}(G)\leq r\cdot k\cdot d^{r-1}nd/r+|H|\cdot d^r\leq (r+1)nkd^r/r$.
Let $T_\text{c}\coloneqq\sum_{G\in\mathcal{F}_{\overline{e}}}T_\text{c}(G)$ be the number of pairs $(G,\Lambda_{\overline{e}})$ which are bad because of \ref{badcond2}.
Then,
\begin{equation}\label{equa:bbad}
T_\text{c}\leq |\mathcal{F}_{\overline{e}}|\frac{(r+1)nkd^r}{r}.
\end{equation}

Finally, we count the number of in-switching configurations which are bad because of \ref{badcond15}.
For this, fix $\Lambda_{\overline{e}}=(f_1,\ldots,f_r,f)\in\bigcup_{G\in\mathcal{F}_{\overline{e}}}\mathcal{L}(G)$.
Note that this implies that $\Lambda_{\overline{e}}\cap H'=\varnothing$.
We now apply \cref{lema:switch1} with $H\cup\Lambda_{\overline{e}}$ playing the role of $H$ and $H'\cup\{e\}$ playing the role of $H'$ to bound the number of pairs $(G,\Lambda_{\overline{e}})$ that are bad because of \ref{badcond15}.
We denote this number by $T_\text{d}$.
\Cref{lema:switch1} implies that, for any $\hat{e}\in\binom{V}{r}\setminus(H\cup H'\cup\Lambda_{\overline{e}}\cup\{e\})$,
\[\mathbb{P}\left[\hat{e}\in G_{n,d}^{(r)}\mid\mathcal{G}_{n,d,H\cup\Lambda_{\overline{e}},H'\cup\{e\}}^{(r)}\right]\leq2(r-1)!\frac{d}{n^{r-1}}\COMMENT{We have $\mathbb{P}\left[\hat{e}\in G_{n,d}^{(r)}\mid\mathcal{G}_{n,d,H\cup\Lambda_{\overline{e}},H'\cup\{e\}}^{(r)}\right]\leq(r-1)!\frac{d}{n^{r-1}}\left(1+\bigO\left(\frac{1}{n}+\frac{d}{n^{r-1}}+\frac{k}{d}+\frac{k'}{n^{r-1}}\right)\right)\leq2(r-1)!\frac{d}{n^{r-1}}$, where the $k/d$ term comes from observing that $|H\cup\Lambda_{\overline{e}}|\leq kn/r+r+1$.}.\]
In particular, this holds for all $r$-sets of the form $(f_i\setminus\{v_i\})\cup\{v\}$ for some $i\in[r]$ and $v\in f$ (as long as they are not in $H$ or $H'$, which is guaranteed for condition \ref{badcond15}).
Therefore, a union bound yields an upper bound on the probability that $\Lambda_{\overline{e}}$ is bad for $G$  because of \ref{badcond15}.
Indeed, let $\mathcal{B}(G,\Lambda_{\overline{e}})$ denote the event that the pair $(G,\Lambda_{\overline{e}})$ is bad because of \ref{badcond15}.
Then\COMMENT{If $\Lambda_{\overline{e}}$ is bad because of a or b, the probability is $0$. Otherwise,},
\begin{equation}\label{equa:refereeadded1}
\mathbb{P}\left[\mathcal{B}(G_{n,d}^{(r)},\Lambda_{\overline{e}})\mid\mathcal{G}_{n,d,H\cup\Lambda_{\overline{e}},H'\cup\{e\}}^{(r)}\right]\leq 2r^2(r-1)!\frac{d}{n^{r-1}}.
\end{equation}
The same approach works for all $\Lambda_{\overline{e}}$.
By \eqref{equa:switch2up2} we have that $|\varOmega|\leq |\mathcal{F}_{\overline{e}}|nd^{r+1}/{r}$.
Moreover, note that 
\begin{equation}\label{equa:refereeadded2}
|\varOmega|=\sum_{\Lambda_{\overline{e}}\in\bigcup_{G\in\mathcal{F}_{\overline{e}}}\mathcal{L}(G)}|\mathcal{G}^{(r)}_{n,d,H\cup\Lambda_{\overline{e}},H'\cup\{e\}}|.
\end{equation}
Hence, for the number $T_\text{d}$ of pairs that are bad because of \ref{badcond15}, by \eqref{equa:refereeadded1} and \eqref{equa:refereeadded2} it follows that
\begin{equation}\label{equa:abadpairs}
T_\text{d}=\sum_{\Lambda_{\overline{e}}\in\bigcup_{G\in\mathcal{F}_{\overline{e}}}\mathcal{L}(G)}|\mathcal{G}^{(r)}_{n,d,H\cup\Lambda_{\overline{e}},H'\cup\{e\}}|\cdot\mathbb{P}\left[\mathcal{B}(G_{n,d}^{(r)},\Lambda_{\overline{e}})\mid\mathcal{G}_{n,d,H\cup\Lambda_{\overline{e}},H'\cup\{e\}}^{(r)}\right]\leq|\mathcal{F}_{\overline{e}}|2r!\frac{d^{r+2}}{n^{r-2}}.
\end{equation}

By \eqref{equa:switch2low2} we have that $|\varOmega|\geq|\mathcal{F}_{\overline{e}}|\frac{nd^{r+1}}{r}\left(1-\bigO\left(\frac{1}{d}+\frac{1}{n}\right)\right)$.
Let $\varepsilon\coloneqq\varepsilon(n,d,k,k')=\frac{1}{n}+\frac{1}{d}+\frac{d}{n^{r-1}}+\frac{k}{d}+\frac{k'}{d}$.
By \eqref{equa:swithc2eq2} and \eqref{equa:abad}--\eqref{equa:abadpairs}, we conclude that
\begin{equation}\label{equa:switch2low}
S_{\text{in}}\geq\lambda_{\text{out}}(|\varOmega|-T_\text{a}-T_\text{b}-T_\text{c}-T_\text{d})=\COMMENT{$\lambda_{\text{out}}(|\varOmega|-T_\text{a}-T_\text{b}-T_\text{c}-T_\text{d})\geq r!\left(|\mathcal{F}_{\overline{e}}|\frac{nd^{r+1}}{r}\left(1-\bigO\left(\frac{1}{d}+\frac{1}{n}\right)\right)-2|\mathcal{F}_{\overline{e}}|r!\frac{d^{r+2}}{n^{r-2}}\right.$ $\left.-|\mathcal{F}_{\overline{e}}|\frac{(r+1)nkd^r}{r}-|\mathcal{F}_{\overline{e}}|rnkd^r-|\mathcal{F}_{\overline{e}}|rnk'd^r\right)=|\mathcal{F}_{\overline{e}}|(r-1)!nd^{r+1}\left(1-\bigO\left(\varepsilon\right)\right)$}|\mathcal{F}_{\overline{e}}|(r-1)!nd^{r+1}\left(1-\bigO\left(\varepsilon\right)\right).
\end{equation}
Combining \eqref{equa:switch2up} and \eqref{equa:switch2low}, we conclude that
\[|\mathcal{F}_{\overline{e}}|(r-1)!nd^{r+1}\left(1-\bigO\left(\varepsilon\right)\right)\leq S_{\text{in}}=S_{\text{out}}\leq|\mathcal{F}_e|(nd)^r.\]
Dividing this by $|\mathcal{G}_{n,d,H,H'}^{(r)}|$ implies that
\[(r-1)!nd^{r+1}\left(1-\bigO\left(\varepsilon\right)\right)\mathbb{P}\left[e\notin G_{n,d}^{(r)}\mid\mathcal{G}_{n,d,H,H'}^{(r)}\right]\leq(nd)^r\mathbb{P}\left[e\in G_{n,d}^{(r)}\mid\mathcal{G}_{n,d,H,H'}^{(r)}\right].\]
Taking into account that $\mathbb{P}[e\notin G_{n,d}^{(r)}\mid\mathcal{G}_{n,d,H,H'}^{(r)}]=1-\mathbb{P}[e\in G_{n,d}^{(r)}\mid\mathcal{G}_{n,d,H,H'}^{(r)}]$, we conclude that
\[\mathbb{P}\left[e\in G_{n,d}^{(r)}\mid\mathcal{G}_{n,d,H,H'}^{(r)}\right]\geq\COMMENT{$\geq\frac{(r-1)!nd^{r+1}\left(1-\bigO\left(\varepsilon\right)\right)}{(nd)^r\left(1+\frac{(r-1)!nd^{r+1}\left(1-\bigO\left(\varepsilon\right)\right)}{(nd)^r}\right)}$}(r-1)!\frac{d}{n^{r-1}}\left(1-\bigO\left(\varepsilon\right)\right).\qedhere\]
\end{proof}

Together, \cref{lema:switch1} and \cref{lema:switch2} imply the following result.

\begin{corollary}\label{coro:switchprob}
Let\/ $r\geq2$ be a fixed integer.
Suppose that\/ $d=\omega(1)$ and\/ $d=o(n^{r-1})$.
Let\/ $H,H'\subseteq\binom{V}{r}$ be two edge-disjoint\/ $r$-graphs such that\/ $\Delta(H),\Delta(H')=o(d)$.
Then, for all\/ $e\in\binom{V}{r}\setminus(H\cup H')$ we have
\[\mathbb{P}\left[e\in G_{n,d}^{(r)}\mid\mathcal{G}_{n,d,H,H'}^{(r)}\right]=(r-1)!\frac{d}{n^{r-1}}\left(1\pm\bigO\left(\frac{1}{n}+\frac{1}{d}+\frac{d}{n^{r-1}}+\frac{\Delta(H)}{d}+\frac{\Delta(H')}{d}\right)\right).\]
\end{corollary}


\section{Counting subgraphs of random regular $r$-graphs}\label{section3}

In this section we use the results of \cref{section2} to count the number of copies of certain $r$-graphs $F$ inside a random $d$-regular $r$-graph.
In \cref{section31} we consider the case when $F$ is fixed.
In particular, we will derive results on the number of edge-disjoint copies of $F$, which will be used in \cref{section42}.
In \cref{section32} we apply our results to count the number of copies of sparse but possibly spanning $r$-graphs\COMMENT{Results extend to almost spanning and, in general, to any subgraph with the needed conditions} such as Hamilton cycles.


\subsection{Counting small subgraphs}\label{section31}

For an $r$-graph $F$, let $\mathrm{aut}(F)$ denote the number of automorphisms of $F$.
Let $X_F(G)$ denote the number of (unlabelled) copies of $F$ in an $r$-graph $G$.
We will often just write $X_F$ whenever $G$ is clear from the context.
Observe that $X_F$ is a random variable whenever $G$ is randomly chosen from some set $\mathcal{G}$.
We will consider the uniform distribution on the set $\mathcal{G}_{n,d}^{(r)}$.
Furthermore, we define \[p\coloneqq(r-1)!\frac{d}{n^{r-1}}\qquad\text{ and }\qquad\varepsilon_{n,d}\coloneqq\frac{1}{n}+\frac{1}{d}+\frac{d}{n^{r-1}}.\]

\begin{corollary}\label{coro:expectation}
Let\/ $r\geq2$ and\/ $t\geq1$ be fixed integers, and let\/ $F$ be a fixed\/ $r$-graph.
Suppose that\/ $d=\omega(1)$ and\/ $d=o(n^{r-1})$.
Then, 
\begin{enumerate}[label=(\roman*)]
\item for any set\/ $\mathcal{E}\subseteq\binom{V}{r}$ of size\/ $t$,\/ $\displaystyle\mathbb{P}[\mathcal{E}\subseteq G_{n,d}^{(r)}]=p^t\left(1\pm\bigO\left(\varepsilon_{n,d}\right)\right)$,\label{item1}
\item $\displaystyle\mathbb{E}[X_F]=\binom{n}{v_F}\frac{v_F!}{\mathrm{aut}(F)}p^{e_F}\left(1\pm\bigO\left(\varepsilon_{n,d}\right)\right)$.\label{item3}
\end{enumerate}
\end{corollary}

\begin{proof}
Enumerate the edges in $\mathcal{E}$ as $e_1,\ldots,e_t$.
\ref{item1} follows by applying \cref{coro:switchprob} repeatedly\COMMENT{.
In the $i$-th application we apply \cref{coro:switchprob} with an $r$-graph $H_i$ playing the role of $H$, where $H_i$ is defined by setting $H_1\coloneqq\varnothing$ and $H_i\coloneqq H_{i-1}\cup\{e_{i-1}\}$ for $i\in[t]\setminus\{1\}$}.
This in turn implies \ref{item3}\COMMENT{
In order to prove \ref{item3}, note that the number of copies $F'$ of $F$ with $F'\subseteq\binom{V}{r}$ is $\binom{n}{v_F}\frac{v_F!}{\mathrm{aut}(F)}$.
Now, the result follows by linearity of expectation and \ref{item1}.}.
\end{proof}

The next lemma implies that $X_F$ is concentrated around $\mathbb{E}[X_F]$ whenever $\Phi_F=\omega(1)$, where
\[\Phi_{F}\coloneqq\min\{\mathbb{E}[X_K]:K\subseteq F, e_K>0\}.\]

\begin{lemma}\label{lema:variance}
Let\/ $r\geq2$ be a fixed integer.
Suppose that\/ $d=\omega(1)$ and\/ $d=o(n^{r-1})$.
Then, for any fixed\/ $r$-graph\/ $F$ with\/ $e_F\geq1$, we have that\/ $\mathrm{Var}[X_F]=\bigO(\varepsilon_{n,d}+\Phi_{F}^{-1})\mathbb{E}[X_F]^2$.
\end{lemma}

The proof follows a straightforward second moment approach (based on \cref{coro:expectation}), so we omit the details (for a proof of the same statement in $\mathcal{G}_{n,p}$, see for instance \cite[Lemma~3.5]{RGbook00})\COMMENT{
In this proof we write $\varepsilon\coloneqq\varepsilon_{n,d}$.
Let \[\mathcal{F}\coloneqq\bigg\{F'\subseteq K_V^{(r)}:F'\cong F\bigg\}.\]
Given a set $A$, we define the function $\mathit{sgn}(A)$ by $\mathit{sgn}(\varnothing)=0$ and $\mathit{sgn}(A)=1$ for all $A\neq\varnothing$.
For $i\in\{0,1\}$, we define
\begin{align*}
\mathcal{F}_i\coloneqq\{(F',F'')\in\mathcal{F}^2:&\mathit{sgn}(E(F')\cap E(F''))=i\}
\end{align*}
and
\[Y_i=\sum_{(F',F'')\in\mathcal{F}_i}\left(\mathbb{E}[I_{F'\cup F''}]-\mathbb{E}[I_{F'}]\mathbb{E}[I_{F''}]\right),\]
where $I_{F'}$ is the indicator random variable that a fixed copy $F'$ of $F$ appears in the random $r$-graph; that is, $X_F=\sum_{F'\in\mathcal{F}}I_{F'}$.\\
We first bound $Y_0$.
For $k\in[v_F]\cup \{0\}$, define $\widehat{\mathcal{F}}(k)\coloneqq\{F_1\cup F_2:F_1,F_2\in\mathcal{F}, E(F_1)\cap E(F_2)=\varnothing, |V(F_1)\cap V(F_2)|=k\}$.
By applying \cref{coro:expectation}\ref{item1} and \ref{item3}, we compute
\begin{align}\label{equa:var6}
Y_0&=\sum_{k=0}^{v_F}\sum_{\widehat F\in\widehat{\mathcal{F}}(k)}\sum_{\substack{(F',F'')\in\mathcal{F}_0\\F'\cup F''\cong\widehat F}}\left(p^{2e_F}\left(1\pm\bigO\left(\varepsilon\right)\right)-\left(p^{e_F}\left(1\pm\bigO\left(\varepsilon\right)\right)\right)^{2}\right)\nonumber\\
&=\sum_{k=0}^{v_F}\sum_{\widehat F\in\widehat{\mathcal{F}}(k)}\bigO(\varepsilon n^{2v_F-k}p^{2e_F})=\bigO(\varepsilon\mathbb{E}[X_F]^2).
\end{align}
We now bound $Y_1$.
Let $K\subseteq F$.
Then, there are $\bigO\left(n^{2v_F-v_K}\right)$ pairs $(F',F'')\in\mathcal{F}_1$ such that $F'\cap F''\cong K$.
Thus, by \cref{coro:expectation},
\begin{align}\label{equa:var8}
Y_1\leq&\sum_{(F',F'')\in\mathcal{F}_1}\mathbb{E}[I_{F'\cup F''}]=\sum_{\substack{K\subseteq F\\e_K>0}}\sum_{\substack{K'\subseteq K_V^{(r)}\\K'\cong K}}\sum_{\substack{(F',F'')\in\mathcal{F}_1\\F'\cap F''=K'}}p^{2e_F-e_K}\left(1\pm\bigO\left(\varepsilon\right)\right)\nonumber\\
=&\sum_{\substack{K\subseteq F\\e_K>0}}\bigO\left(n^{2v_F-v_K}p^{2e_F-e_K}\right)=\mathbb{E}[X_F]^2\sum_{\substack{K\subseteq F\\e_K>0}}\bigO\left((n^{v_K}p^{e_K})^{-1}\right)=\bigO\left(\Phi_{F}^{-1}\mathbb{E}[X_F]^2\right),
\end{align}
where the final equality holds by the definition of $\Phi_F$.
We have that
\[\mathrm{Var}[X_F]=\sum_{(F',F'')\in\mathcal{F}_0\cup\mathcal{F}_1}(\mathbb{E}[I_{F'}I_{F''}]-\mathbb{E}[I_{F'}]\mathbb{E}[I_{F''}])=Y_0+Y_1.\]
Combining \eqref{equa:var6} and \eqref{equa:var8} proves the statement.}.
\Cref{coro:expectation}, \cref{lema:variance} and Chebyshev's inequality imply the following result.
In particular, this determines the threshold for the appearance of a copy of a fixed $F$ in $\mathcal{G}_{n,d}^{(r)}$.

\begin{corollary}\label{coro:copies}
Let\/ $r\geq2$ be a fixed integer.
Suppose that\/ $d=\omega(1)$ and\/ $d=o(n^{r-1})$.
Then, for any fixed\/ $r$-graph\/ $F$ with\/ $\Phi_{F}=\omega(1)$, we a.a.s.~have
\[X_F=\left(1\pm o(1)\right)\binom{n}{v_F}\frac{v_F!}{\mathrm{aut}(F)}p^{e_F}.\]
\end{corollary}

The next result adresses the problem of counting edge-disjoint copies of an $r$-graph $F$ in $G_{n,d}^{(r)}$.
Its proof builds on an idea of \citet{Kreuter96} for counting vertex-disjoint copies in the binomial random graph model (see also \cite[Theorem 3.29]{RGbook00}).
The approach is to consider an auxiliary graph whose vertex set consists of the copies of $F$ in $G_{n,d}^{(r)}$ and where an independent set corresponds to a set of edge-disjoint copies of $F$.
To estimate the number of vertices and edges of this graph (with a view to apply Tur\'an's theorem), one makes use of \cref{coro:expectation}, \cref{lema:variance,coro:copies}.
For the sake of completeness, we include the details in \cref{appendix1}.

\begin{lemma}\label{lema:disjointcopies}
Let\/ $F$ be a fixed\/ $r$-graph.
Assume that\/ $d=\omega(1)$ and\/ $d=o(n^{r-1})$.
Let\/ $D_F$ be the maximum number of edge-disjoint copies of\/ $F$ in an $r$-graph chosen uniformly from\/ $\mathcal{G}_{n,d}^{(r)}$.
If\/ $\Phi_{F}=\omega(1)$, then\/ $D_F=\Theta(\Phi_{F})$ a.a.s.
\end{lemma}


\subsection{Counting spanning graphs}\label{section32}

Let $H=\{H_i\}_{i\geq1}$ be a sequence of $r$-graphs with $|V(H_i)|$ strictly increasing.
When we say that $H$ is a subgraph of $G$, for some $G$ of order $n$, we mean that the corresponding $H_i$ of order $n$ is a subgraph of $G$.
This only makes sense when $n=|V(H_i)|$ for some $i$; we will implicitly assume this is the case, and study the asymptotic behaviour as $i$ tends to infinity.

Our main tool for this section is the following result of \citet{DFRSsandwitch}, which allows to translate results on the $\mathcal{G}^{(r)}(n,p)$ and $\mathcal{G}^{(r)}(n,m)$ random graph models to $\mathcal{G}_{n,d}^{(r)}$.
Roughly speaking, their result asserts that $G^{(r)}(n,p)\subseteq G_{n,d}^{(r)}$ a.a.s.~provided that $p$ is at least a little smaller than $d/\binom{n-1}{r-1}$.
For the graph case, a similar result was proved by \citet{KVsandwich} (for a more restricted range of $d$).

\begin{theorem}[\cite{DFRSsandwitch}]\label{teor:sandwichnm}
For every\/ $r\geq2$ there exists a constant\/ $C>0$ such that if for some positive integer\/ $d=d(n)$,
\begin{equation}
\delta_{n,d}\coloneqq C\left(\left(\frac{d}{n^{r-1}}+\frac{\log n}{d}\right)^{1/3}+\frac{1}{n}\right)<1,
\end{equation}
then there is a joint distribution of\/ $G^{(r)}(n,p_d)$ and\/ $G_{n,d}^{(r)}$ such that
\[\lim_{n\to\infty}\mathbb{P}\left[G^{(r)}(n,p_d)\subseteq G_{n,d}^{(r)}\right]=1,\]
where\/ $p_d\coloneqq(1-\delta_{n,d})d/\binom{n-1}{r-1}$.
The analogous statement also holds with\/ $G^{(r)}(n,p_d)$ replaced by\/ $G^{(r)}(n,m_d)$ for\/ $m_d\coloneqq(1-\delta_{n,d})nd/r$\COMMENT{If we worry about $m_d$ being an integer, then this statement is weaker than the original one.}.
\end{theorem}

In order to be able to apply \cref{teor:sandwichnm}, from now on we always assume that $d=o(n^{r-1})$ and $d=\omega(\log n)$.
We now combine \cref{teor:sandwichnm} with our results from \cref{section2} to obtain a general result relating subgraph counts in $\mathcal{G}_{n,d}^{(r)}$ to those in $G^{(r)}(n,p_d)$ and $G^{(r)}(n,m_d)$\COMMENT{A variant of \cref{teor:generalspanning} also holds when $\Delta(H)$ is allowed to be as large as $o(d)$.
Here we only state this version, as it is the version we will use in our applications. Note that we make no effort to improve the constants.}.

\begin{theorem}\label{teor:generalspanning}
Let\/ $r\geq2$ be a fixed integer and\/ $V$ be a set of\/ $n$ vertices.
Assume that\/ $d=\omega(\log n)$ and\/ $d=o(n^{r-1})$.
Let\/ $H$ be an\/ $r$-graph on\/ $V$ with\/ $\Delta(H)=\bigO(1)$.
Suppose that\/ $\eta=\eta(n)=o(1)$ is such that 
\begin{equation}\label{equa:spancond}
\varepsilon_{n,d}=o(\eta),\qquad \delta_{n,d}=o(\eta),\qquad \eta=\omega(1/n),
\end{equation}
and\/ $X_H(G^{(r)}(n,p_d))=(1\pm\eta)^{|H|}\mathbb{E}[X_H(G^{(r)}(n,p_d))]$ a.a.s. 
Then a.a.s.
\begin{equation}\label{equa:spanstate1}
X_H(G_{n,d}^{(r)})=(1\pm3\eta)^{|H|}\mathbb{E}[X_H(G^{(r)}(n,p_d))]
\end{equation}
Similarly, if \eqref{equa:spancond} holds and\/ $X_H(G^{(r)}(n,m_d))=(1\pm\eta)^{|H|}\mathbb{E}[X_H(G^{(r)}(n,m_d))]$ a.a.s., then a.a.s.
\begin{equation}\label{equa:spanstate2}
X_H(G_{n,d}^{(r)})=(1\pm3\eta)^{|H|}\mathbb{E}[X_H(G^{(r)}(n,m_d))].
\end{equation}
\end{theorem}

\begin{proof}
Observe first that, by \cref{coro:switchprob}, for any fixed copy $H'$ of $H$ we have\COMMENT{As in the proof of \cref{coro:expectation}, this follows by using \cref{coro:switchprob} repeatedly, adding an edge to the conditioned graph in each step.}
\begin{equation}\label{equa:span0}
\mathbb{P}\left[H'\subseteq G_{n,d}^{(r)}\right]=((1\pm\bigO(\varepsilon_{n,d}))(r-1)!d/n^{r-1})^{|H|}.
\end{equation}
Therefore\COMMENT{We have that\begin{align*}
\frac{\mathbb{E}[X_H(G_{n,d}^{(r)})]}{\mathbb{E}[X_H(G^{(r)}(n,p_d))]}&=\frac{\sum_{H'\cong H}\mathbb{P}[H'\subseteq G_{n,d}^{(r)}]}{\sum_{H'\cong H}\mathbb{P}[H'\subseteq G^{(r)}(n,p_d)]}\\
&=\frac{|\{H'\subseteq K_V^{(r)}:H'\cong H\}|((1\pm\bigO(\varepsilon_{n,d}))(r-1)!d/n^{r-1})^{|H|}}{|\{H'\subseteq K_V^{(r)}:H'\cong H\}|((1-\delta_{n,d})d/\binom{n-1}{r-1})^{|H|}}\\
&=\frac{((1\pm\bigO(\varepsilon_{n,d}))(r-1)!d/n^{r-1})^{|H|}}{((1+\bigO(1/n))(r-1)!(1-\delta_{n,d})d/n^{r-1})^{|H|}}\\
&=\frac{(1\pm\bigO(\varepsilon_{n,d}))^{|H|}}{((1+\bigO(1/n))(1-\bigO(\delta_{n,d})))^{|H|}}=(1\pm\bigO(\varepsilon_{n,d})+\bigO(\delta_{n,d}))^{|H|}.
\end{align*}},
\begin{equation}\label{equa:span1}
\frac{\mathbb{E}[X_H(G_{n,d}^{(r)})]}{\mathbb{E}[X_H(G^{(r)}(n,p_d))]}=(1\pm\bigO(\varepsilon_{n,d}+\delta_{n,d}))^{|H|}\leq(1+\eta)^{|H|}.
\end{equation}
By using Markov's inequality and \eqref{equa:span1} we conclude that
\begin{align}\label{equa:span2}
&\mathbb{P}\left[X_H(G_{n,d}^{(r)})\geq(1+3\eta)^{|H|}\mathbb{E}\left[X_H(G^{(r)}(n,p_d))\right]\right]\nonumber\\
\leq\,&\mathbb{P}\left[X_H(G_{n,d}^{(r)})\geq(1+\eta)^{|H|}\mathbb{E}\left[X_H(G_{n,d}^{(r)})\right]\right]\leq1/(1+\eta)^{|H|}=o(1).
\end{align}
Note that, as $G^{(r)}(n,p_d)\subseteq G_{n,d}^{(r)}$ a.a.s.~by \cref{teor:sandwichnm}, then $X_H(G_{n,d}^{(r)})\geq X_H(G^{(r)}(n,p_d))$ a.a.s.
Thus, by assumption,
\begin{align}\label{equa:span3}
&\mathbb{P}\left[X_H(G_{n,d}^{(r)})\leq(1-\eta)^{|H|}\mathbb{E}\left[X_H(G^{(r)}(n,p_d))\right]\right]\nonumber\\
\leq\,&\mathbb{P}\left[X_H(G^{(r)}(n,p_d))\leq(1-\eta)^{|H|}\mathbb{E}\left[X_H(G^{(r)}(n,p_d))\right]\right]+o(1)=o(1).
\end{align}
Combining equations \eqref{equa:span2} and \eqref{equa:span3} yields \eqref{equa:spanstate1}.

Finally, one can prove \eqref{equa:spanstate2} in a very similar way\COMMENT{
Consider a general $r$-graph $H$ and general value $m$ and let $d=rm/n$.
Write $h=|H|$, $N=\binom{n}{r}$ and $p=m/N=d/\binom{n-1}{r-1}$.
Then
\begin{align*}
\mathbb{P}[H \subseteq G^{(r)}(n,m)] & =  \binom{h}{h} \binom{ N -h }{m-h} / \binom{N}{m} = \frac{(N-h)_{m-h} }{(m-h)!} \frac{m!}{(N)_m} = \frac{(m)_h}{(N)_h} \\
 & \ge  \left( \frac{(m-h)}{m} \frac{m} { N }\right)^h =  \left( \frac{(m-h)}{m} \frac{dn}{r} \frac{r!}{(n)_r} \right)^h = (1-h/m)^h \left( d / \binom{n-1}{r-1} \right)^h = (1-h/m)^hp^h .
\end{align*}
Substituting the value of $m_d$ here,
\begin{align}\label{equa:span4}
\frac{\mathbb{E}[X_H(G_{n,d}^{(r)})]}{\mathbb{E}[X_H(G^{(r)}(n,m_d))]}&=\frac{\sum_{H'\cong H}\mathbb{P}[H'\subseteq G_{n,d}^{(r)}]}{\sum_{H'\cong H}\mathbb{P}[H'\subseteq G^{(r)}(n,m_d)]}\nonumber\\
&\leq\frac{|\{H'\subseteq K_V^{(r)}:H'\cong H\}|((1\pm\bigO(\varepsilon_{n,d}))(r-1)!d/n^{r-1})^{|H|}}{|\{H'\subseteq K_V^{(r)}:H'\cong H\}|((1-|H|/m_d)(1-\delta_{n,d})d/\binom{n-1}{r-1})^{|H|}}\nonumber\\
&=\frac{((1\pm\bigO(\varepsilon_{n,d}))(r-1)!d/n^{r-1})^{|H|}}{((1+\bigO(1/n))(1-|H|/m_d)(1-\delta_{n,d})(r-1)!d/n^{r-1})^{|H|}}\nonumber\\
&=\frac{(1\pm\bigO(\varepsilon_{n,d}))^{|H|}}{((1+\bigO(1/n))(1-\bigO(\delta_{n,d})))^{|H|}}\nonumber\\
&=(1\pm\bigO(\varepsilon_{n,d})+\bigO(\delta_{n,d}))^{|H|}\leq(1+\eta)^{|H|}.
\end{align}
Note that $(1-|H|/m_d)=(1-\bigO(1/d))$ is absorbed by $(1-\bigO(\varepsilon_{n,d}))$.
By using Markov's inequality and \eqref{equa:span4} we conclude that
\begin{align}\label{equa:span5}
&\mathbb{P}[X_H(G_{n,d}^{(r)})\geq(1+3\eta)^{|H|}\mathbb{E}[X_H(G^{(r)}(n,m_d))]]\nonumber\\
\leq\,&\mathbb{P}[X_H(G_{n,d}^{(r)})\geq(1+\eta)^{|H|}\mathbb{E}[X_H(G_{n,d}^{(r)})]]\leq1/(1+\eta)^{|H|}=o(1).
\end{align}
We have that $X_H(G_{n,d}^{(r)})\geq X_H(G^{(r)}(n,m_d))$ a.a.s.~by \cref{teor:sandwichnm}.
Therefore, by assumption,
\begin{align}\label{equa:span6}
&\mathbb{P}[X_H(G_{n,d}^{(r)})\leq(1-\eta)^{|H|}\mathbb{E}[X_H(G^{(r)}(n,m_d))]]\nonumber\\
\leq\,&\mathbb{P}[X_H(G^{(r)}(n,m_d))\leq(1-\eta)^{|H|}\mathbb{E}[X_H(G^{(r)}(n,m_d))]]+o(1)=o(1).
\end{align}
Combining equations \eqref{equa:span5} and \eqref{equa:span6} yields \eqref{equa:spanstate2}.}.
\end{proof}

We may apply \cref{teor:generalspanning} to obtain estimates on the number of copies of certain spanning subgraphs.
This requires concentration results in the $\mathcal{G}^{(r)}(n,p)$ model or the $\mathcal{G}^{(r)}(n,m)$ model in order to obtain results for $\mathcal{G}_{n,d}^{(r)}$.

We start with the following result of \citet{GK13} on counting Hamilton cycles in $\mathcal{G}(n,p)$.
For a more restricted range of densities, \citet{Janson} proved more precise results in $\mathcal{G}(n,m)$.

\begin{theorem}[\cite{GK13}]\label{teor:GleKri}
Let\/ $V$ be a set of\/ $n$ vertices.
Let\/ $H$ be a Hamilton cycle on\/ $V$.
If\/ $p\geq\frac{\ln n+\ln\ln n+\omega(1)}{n}$, then a.a.s.
\[X_H(G(n,p))=(1\pm o(1))^nn!p^{n}.\]
\end{theorem}

Together with \cref{teor:generalspanning} this implies the following result\COMMENT{
The given conditions on $d$ ensure that the conditions from \cref{teor:GleKri} hold for $p_d$, thus $X_H(G^{(r)}(n,p_d))=(1\pm o(1))^{|H|}\mathbb{E}[X_H(G^{(r)}(n,p_d))]$.
The result now follows from \cref{teor:generalspanning}.
}.

\begin{corollary}\label{coro:hamcycleskrive}
Let\/ $V$ be a set of\/ $n$ vertices.
Let\/ $H$ be a Hamilton cycle on\/ $V$.
Assume\/ $d=\omega(\log n)$ and\/ $d=o(n)$, then a.a.s.
\[X_H(G_{n,d})=(1\pm o(1))^nn!\left(\frac{d}{n-1}\right)^{n}.\]
\end{corollary}

\Cref{coro:hamcycleskrive} improves a previous result of \citet{Kri12} by increasing the range of $d$ in which the number of Hamilton cycles is estimated from $d=\omega(e^{(\log n)^{1/2}})$\COMMENT{In \cite{Kri12}, they prove the same estimate as long as $\log d\cdot \log(d/\lambda)\gg\log n$, where $\lambda$ is the second biggest eigenvalue (in absolute value). For random $d$-regular graphs, this is $\Theta(\sqrt{d})$. The range stated above follows.} to $d=\omega(\log n)$.
Note that, on the other hand, the results of \citet{Kri12} also cover pseudo-random $d$-regular graphs.

A very general result due to \citet{Riordan} allows us to count the number of copies of $H$ as a spanning subgraph of $G(n,m)$ for a large class of graphs $H$.
We only state a special case of this result here.
Let $\alpha_1(H)\coloneqq|H|/\binom{n}{2}$, $\alpha_2(H)=X_{P_2}(H)/(3\binom{n}{3})$ (where $P_2$ stands for a path of length $2$), $e_H(k)\coloneqq\max\{|F|:F\subseteq H,|V(F)|=k\}$, $\gamma_1(H)\coloneqq\max_{3\leq k\leq n}\{e_H(k)/(k-2)\}$ and $\gamma_2(H)\coloneqq\max_{5\leq k\leq n}\{(e_H(k)-4)/(k-4)\}$\COMMENT{For this, we need to check many conditions. Some of them are trivial. For the others,\\
Condition 2.1 becomes bound on $p$.\\
Condition 2.2 becomes bound on $p$.\\
Condition 2.5 holds if $p=o(1/\log n)$, as $|H|=\Theta(n)$.\\
Condition 2.6 becomes bound on $p$ ($n^{1/2}$).\\
Condition 2.7: $\alpha_2=\sum\binom{d(v)}{2}/3\binom{n}{3}$. As the maximum degree is constant, the sum is linear and the whole thing is $\bigO(1/n^2)$. $\alpha_1=\Theta(1/n)$, hence second part of condition holds.
The first part is added to the statement.
Note that Riordan writes $\alpha_2(H)-\alpha_1(H)^2=\Omega(1/n^2)$ to mean that $|\alpha_2(H)-\alpha_1(H)^2|=\Omega(1/n^2)$, see the remark after Thm. 2.2 in his paper.\\
Condition 2.2' cannot hold, so we need condition 2.8. This becomes bound on $p$.}.

\begin{theorem}[\cite{Riordan}]\label{teor:riordan}
Let\/ $V$ be a set of\/ $n$ vertices.
Let\/ $p=\omega(\max\{1/n^{1/2},1/n^{1/\gamma_1},1/n^{1/\gamma_2}\})$,\/ $p=o(1/\log n)$,\/  $m\coloneqq p\binom{n}{2}$, and let\/ $H$ be a triangle-free spanning graph on\/ $V$ with\/ $|H|\geq n$,\/ $\Delta(H)=\bigO(1)$ and\/ $|\alpha_2(H)-\alpha_1(H)^2|=\Omega(1/n^2)$.
Then,\/ $X_H(G(n,m))$ follows a normal distribution such that\/ $\mathrm{Var}[X_H(G(n,m))]/\mathbb{E}[X_H(G(n,m))]^2=o(1)$.
\end{theorem}

Together with \cref{teor:generalspanning}, we can deduce the following.\COMMENT{Proof: It is easy to check that the conditions in the statement of \cref{teor:riordan} hold when taking $p=(1-\delta_{n,d})d/(n-1)$.
Therefore, \cref{teor:riordan} holds for this value of $p$ and $X_H(G(n,m_d))=\mathbb{E}[X_H(G(n,m_d))](1\pm\eta)^{|H|}$ a.a.s.
The result follows by \cref{teor:generalspanning}.}

\begin{corollary}\label{coro:riordan}
Let\/ $V$ be a set of\/ $n$ vertices.
Assume that\/ $d=\omega(\max\{n^{1/2},n^{1-1/\gamma_1},n^{1-1/\gamma_2}\})$,\/ $d=o(n/\log n)$, and let\/ $H$ be a triangle-free spanning graph on\/ $V$ with\/ $|H|\geq n$,\/ $\Delta(H)=\bigO(1)$ and\/ $|\alpha_2(H)-\alpha_1(H)^2|=\Omega(1/n^2)$.
Then,\/ $X_H(G_{n,d})=(1\pm o(1))^{n}\mathbb{E}[X_H(G(n,m_d))]$ a.a.s., where\/ $m_d=(1-o(1))dn/2$ is defined as in \cref{teor:generalspanning}.
\end{corollary}

As a particular case of this, we can estimate the number of spanning square lattices in a random $d$-regular graph.
A square lattice $L_k$ is defined by setting $V(L_k)=[k]\times[k]$ and $L_k=\{\{(x,y),(u,v)\}:u,v,x,y\in[k],\lVert(x,y)-(u,v)\rVert=1\}$.

\begin{corollary} \label{coro:lattice}
Let\/ $n=k^2$.
Let\/ $d=\omega(1)$, \/$d=o(n/\log n)$ and\/ $p\coloneqq d/(n-1)$.
\begin{enumerate}[label=(\roman*)]
\item If\/ $d=o(n^{1/2})$, then\/ $\mathbb{P}[X_{L_k}(G_{n,d})>0]=o(1)$\COMMENT{Observe that by \cref{coro:switchprob} (which needs $d\to\infty$)
\begin{align*}
\mathbb{E}[X_{L_k}(G_{n,d})]&=(1+o(1))^n(n/e)^np^{2n-\bigO(n^{1/2})}\\
&\leq (1/2)^{n}n^n(o(1)/{n}^{1/2})^{2n-\bigO(n^{1/2})}=(o(1))^{2n-o(n)}(1/2)^{n}({n}^{1/2})^{\bigO(n^{1/2})}\leq(1/2)^{n}e^{\bigO(n^{1/2}\log n)}\to0.
\end{align*}
The statement follows by Markov's inequality.}.\label{item1lattice}
\item If\/ $d=\omega(n^{1/2})$\COMMENT{The condition on $\alpha_2$ and $\alpha_1$ holds for $L_k$. Indeed, it is clear that the number of edges is approximately $2n$, thus $\alpha_1\sim4/n$. Similarly, $6(k-2)^2\leq\# P_2\leq 6k^2$, hence $\alpha_2\sim12/n^2$. Moreover, $\gamma_1(L_k)=2$ and $\gamma_2(L_k)\leq2$ (see the proof in \cite[Theorem~1.4]{Riordan}). The bound follows.}, then,\/ $X_{L_k}(G_{n,d})=(1\pm o(1))^nn!p^{|L_k|}$ a.a.s.\label{item2lattice}
\end{enumerate}
\end{corollary}

In particular, as $|L_k|=2n\pm\bigO(n^{1/2})$, this determines the threshold for the existence of a spanning square lattice $L_k$ in $G_{n,d}$.
\cref{coro:lattice}(i) follows from \cref{coro:switchprob} and Markov's inequality,
while \cref{coro:lattice}(ii) follows from \cref{coro:riordan}.\COMMENT{The $0$ statement is given in \cref{coro:lattice}(i). To see the $1$ statement, observe that
\[
\mathbb{E}[X_{L_k}(G_{n,d})]\geq(n/3)^np^{2n-\bigO(n^{1/2})}\geq(n/3)^np^{2n}=(n/3)^n(\omega(1)/n^{1/2})^{2n}=(\omega(1))^{2n}\to\infty.
\]
By \ref{item2lattice}, $X_{L_k}(G_{n,d})\to\infty$ a.a.s.}

Much less is known for $r$-graphs when $r\geq3$.
For Hamilton cycles, we can apply the following result of \citet{DF13} on $\ell$-overlapping Hamilton cycles\COMMENT{These results do not appear in their paper, but can be derived from what they show.}.

\begin{theorem}[\cite{DF13}, Section 2]\label{teor:DuFri}
Let\/ $r>\ell\geq2$ and assume that\/ $(r-\ell)\mid n$.
Assume\/ $p=\omega(1/n^{r-\ell})$.
Then, a.a.s.
\[X_{C_n^\ell}(G^{(r)}(n,p))=(1\pm o(1))^nn!p^{n/(r-\ell)}.\]
\end{theorem}

Together with \cref{teor:generalspanning}, \cref{coro:switchprob} and Markov's inequality, this implies the following result.

\begin{corollary}\label{coro:overlappingcycles}
Let\/ $r>\ell\geq2$ and assume that\/ $(r-\ell)\mid n$.
Let\/ $p\coloneqq d/\binom{n-1}{r-1}$.
\begin{enumerate}[label=(\roman*)]
\item If\/ $d=o(n^{\ell-1})$ then\/ $\mathbb{P}[X_{C_n^\ell}(G^{(r)}_{n,d})>0]=o(1)$\COMMENT{Proof: This follows from \citet{AGIR16} if $d$ is constant.
So assume that $d=\omega(1)$.
The length of an $\ell$-overlapping Hamilton cycle is $n/(r-\ell)$. The expected number of $\ell$-overlapping Hamilton cycles satisfies
\[\mathbb{E}[X_{C_n^\ell}(G^{(r)}_{n,d})]\le\left((1+o(1))p\right)^{n/(r-\ell)} \cdot \frac{n!}{2n}.\]
(see \cite{AGIR16} for the exact formula). Indeed, the probability that a fixed copy of $C_n^\ell$ is present in $G^{(r)}_{n,d}$ is $\left((1+o(1))p\right)^{n/(r-\ell)}$ by repeated applications of \cref{coro:switchprob} (note that \cref{coro:expectation}\ref{item3} does not apply here). Moreover, the number of (undirected) cyclic orderings of $[n]$ is $\frac{n!}{2n}$. Each cyclic ordering gives rise to a copy of $C_n^\ell$ (note that some copy of $C_n^\ell$ might arise from several such orderings). Now suppose $p=o(n^{\ell-r})$. Then by Stirling's formula, $\mathbb{E}[X_{C_n^\ell}(G^{(r)}_{n,d})] \le (2p)^{n/{(r-\ell)}} (n/e)^n\le (1/2)^n=o(1)$. The statement follows by Markov's inequality.}.\label{hamitem1}
\item If\/ $d=\omega(n^{\ell-1})$ and\/ $d=o(n^{r-1})$, then a.a.s.\/
$X_{C_n^\ell}(G^{(r)}_{n,d})=(1\pm o(1))^{n}n!p^{n/(r-\ell)}$\COMMENT{Proof: The given conditions on $d$ ensure that the conditions from \cref{teor:DuFri} hold for $p_d$, thus $X_H(G^{(r)}(n,p_d))=(1\pm o(1))^{|H|}\mathbb{E}[X_H(G^{(r)}(n,p_d))]$. The result now follows from \cref{teor:generalspanning}.}\COMMENT{Note that the existence of a Hamilton cycle for $d=\Theta(n^{r-1})$ is covered by \citet{DFRSsandwitch}. Also, 
\[\left((1+o(1))p\right)^{n/(r-\ell)} \cdot \frac{n!}{2n}\geq(n/(2e))^n(d/n^{r-1})^{n/(r-\ell)}\geq(\omega(1))^{n/(r-\ell)}\to\infty.\]}.
\end{enumerate}
\end{corollary}

In particular, this determines the threshold for the existence of $C_n^\ell$ in $\mathcal{G}_{n,d}^{(r)}$ for $\ell\in[r-1]\setminus\{1\}$, solving a conjecture of \citet{DFRSsandwitch}.
We note that \citet{AGIR16} recently determined the threshold for the appearance of loose Hamilton cycles in random regular $r$-graphs.
Their results imply that for every $r\geq3$ there exists a value $d_0$ (which is calculated explicitly in \cite{AGIR16}) such that if $d\geq d_0$, then $G_{n,d}^{(r)}$ a.a.s.~has a loose Hamilton cycle.
For $\ell\in[r-1]\setminus\{1\}$, they also proved that $\mathbb{P}[X_{C_n^\ell}(G^{(r)}_{n,d})>0]=o(1)$ holds under the much stronger condition that $d=o(n)$ if $r\geq4$ and $d=o(n^{1/2})$ if $r=3$ (but to deduce \cref{coro:overlappingcycles}\ref{hamitem1} we do rely on their result when $d$ is constant; we rely on \cref{coro:switchprob} when $d=\omega(1)$).


\section{Testing $F$-freeness in general $r$-graphs}\label{section4}

We now give lower and upper bounds on the query complexity of testing $F$-freeness in the general $r$-graphs model, where $F$ is a fixed $r$-graph.
In the special case when $F$ is a triangle, these (and other) bounds were already obtained by \citet{AlonAl08}.
Our proofs develop ideas from their paper.

In \cref{section41}, we observe a simple lower bound for the query complexity of any $F$-freeness tester.
In \cref{section42}, we use our results from \cref{section2,section3} to improve this bound for input $r$-graphs whose density is larger than a certain threshold.
The bound that we obtain, however, only holds for one-sided error testers; extending it to two-sided error testers, as \citet{AlonAl08} do with their triangle-freeness tester, would be an interesting problem.
Finally, \cref{section43} is devoted to upper bounds on the query complexity.


\subsection{A lower bound for sparser $r$-graphs}\label{section41}

In this section we provide a lower bound on the query complexity of testing $F$-freeness which is stronger than that in \cref{section42} when the $r$-graphs that are being tested are sparser (the range of the average degree $d$ for which this holds depends on the particular $r$-graph $F$).
Recall that our algorithms are allowed to perform two types of queries: vertex-set queries and neighbour queries.
For a fixed $r$-graph $F$, let $\mathit{ex}(n,F)$ denote the maximum number of edges of an $F$-free $r$-graph $G$ on $n$ vertices.

\COMMENT{\Cref{teor:bound2} does not apply to weak forests $F$ since $\mathit{ex}(n,F)=\Theta(n)$, i.e., $a\leq1$. For weak forests $F$, if the average degree of the input graph is $d=\omega(1)$, there is a trivial tester. The reason is that every $r$-graph $G$ with average degree $d$ has a subgraph of minimum degree at least $d/r$. Indeed, build such a subgraph greedily by deleting in turn a vertex that has degree less than $d/r$. This decreases the sum of the degrees by less than $d$, and the new graph has one less vertex, hence the average degree after each deletion does not decrease. This process has to end at some point. Using this, a weak forest can be embedded greedily into such a subgraph.}

\begin{proposition}\label{teor:bound2}
Let\/ $r\geq2$ and\/ $F$ be an\/ $r$-graph.
Let\/ $c,a>0$ be fixed constants such that\/ $c\cdot n^{a}\leq\mathit{ex}(n,F)$ and suppose that\/ $d=\Omega(1)$ and\/ $d=o(n^{a-1})$.
Then, any\/ $F$-freeness tester in\/ $r$-graphs must perform\/ $\Omega\left(n^{1-1/a}d^{-1/a}\right)$ queries, when restricted to input\/ $r$-graphs on\/ $n$ vertices of average degree\/ $d\pm o(d)$.
\end{proposition}

Observe that the assumptions in the statement imply that $1<a\leq r$.
In particular, the result only applies for $r$-graphs $F$ such that $\mathit{ex}(n,F)$ is superlinear.

\begin{proof}
It suffices to construct two families of $r$-graphs on $n$ vertices $\mathcal{F}_1$ and $\mathcal{F}_2$ such that the following hold:
\begin{enumerate}[label=(\roman*)]
\item All $r$-graphs in $\mathcal{F}_1$ are $F$-free.\label{cond2.1}
\item All $r$-graphs in $\mathcal{F}_2$ are $\Theta(1)$-far from $F$-free.\label{cond2.2}
\item All $r$-graphs in both families have average degree $d\pm o(d)$.\label{cond2.3}
\item Consider an $r$-graph $G$ chosen from $\mathcal{F}_1\cup\mathcal{F}_2$ according to the following rule.
First choose $i\in[2]$ uniformly at random.
Then choose $G\in\mathcal{F}_i$ uniformly at random.
Then any algorithm that determines with probability at least $2/3$ whether $G\in\mathcal{F}_1$ or $G\in\mathcal{F}_2$ must perform at least $\Omega(n^{1-1/a}d^{-1/a})$ queries.\label{cond2.4}
\end{enumerate}

Let $H$ be an $F$-free $r$-graph on $(nd/(cr))^{1/a}$ vertices with $nd/r$ edges\COMMENT{Such an $r$-graph exists since $nd/r=c((nd/(cr))^{1/a})^a\leq\mathit{ex}((nd/(cr))^{1/a},F)$.}.
Let $\mathcal{F}_1$ be the family of all labelled $r$-graphs consisting of the disjoint union of $H$ on $(nd/(cr))^{1/a}$ vertices and $n-(nd/(cr))^{1/a}$ isolated vertices.
Let $\mathcal{F}_2$ be the family of all labelled $r$-graphs consisting of the disjoint union of a complete $r$-graph on a set of $({nd(r-1)!})^{1/r}$ vertices and $n-({nd(r-1)!})^{1/r}$ isolated vertices.

A simple computation shows that all $r$-graphs in both families have average degree $d\pm o(d)$\COMMENT{Graphs in $\mathcal{F}_1$ have $nd/r$ edges by definition, and hence, average degree $d$. On the other hand, \begin{align*}
\binom{(nd(r-1)!)^{1/r}}{r}r/n&=(nd(r-1)!)^{1/r}((nd(r-1)!)^{1/r}-1)\ldots((nd(r-1)!)^{1/r}-(r-1))/(n(r-1)!)\\
&=nd(r-1)!(1-\bigO((nd)^{-1/r}))/(n(r-1)!)=d\pm o(d).
\end{align*}}.
All $r$-graphs in $\mathcal{F}_1$ are $F$-free by definition.
Since the number of distinct $K_{v_F}^{(r)}$ in $K_k^{(r)}$ is $\Theta(k^{v_F})$\COMMENT{This is aymptotic in $k$, it should be obvious from the context.}, it is easy to check\COMMENT{Number of $K_{v_F}^{(r)}$ in $K_k^{(r)}$: $\Theta(k^{v_F})$. Number of $K_{v_F}^{(r)}$ destroyed by one edge: $\bigO(k^{v_F-r})$. Need to delete at least $\Theta(k^r)$ edges.} that all $r$-graphs in $\mathcal{F}_2$ are $\Theta(1)$-far from being $K_{v_F}^{(r)}$-free, and hence $\Theta(1)$-far from being $F$-free.
Thus, conditions \ref{cond2.1}, \ref{cond2.2} and \ref{cond2.3} hold.

Now consider any algorithm ALG that, given an $r$-graph $G$ chosen at random from either $\mathcal{F}_1$ or $\mathcal{F}_2$ as in \ref{cond2.4}, tries to determine with probability at least $2/3$ whether $G\in\mathcal{F}_1$ or $G\in\mathcal{F}_2$.
If $G\in\mathcal{F}_1$, then the probability of finding a vertex with positive degree with any given query is $\bigO(n^{1/a-1}d^{1/a})$\COMMENT{This is $\bigO$ and not $\Theta$ because for a vertex-set query the probability of getting a positive answer, even if one of the vertices involved has positive degree, is smaller (for such a query to have a positive answer, all $r$ vertices would need to lie in $H$).}.
Similarly, if $G\in\mathcal{F}_2$, the probability of finding a vertex with positive degree with any given query is $\bigO(n^{1/a-1}d^{1/a})$\COMMENT{because the number of vertices with positive degree in an $r$-graph in $\mathcal{F}_2$ is $({nd(r-1)!})^{1/r}\leq(nd/(cr))^{1/a}=|V(H)|$, as in $\mathcal{F}_2$ the non-empty subgraph is a clique, which maximises the density, and the number of edges in $r$-graphs in $\mathcal{F}_1$ and $\mathcal{F}_2$ is roughly the same.}.
Hence, if the number of queries is $Q=o(n^{1-1/a}d^{-1/a})$, by the union bound, one has that the probability of finding any such vertex is $o(1)$.
So a.a.s.~ALG only finds a set of isolated vertices, of size $\bigO(Q)$, after the first $Q$ queries.
Thus we conclude that, for $i\in[2]$, $\mathbb{P}[G\in\mathcal{F}_i\mid\text{ALG finds only isolated vertices}]=1/2\pm o(1)$\COMMENT{To be precise, we have that $\mathbb{P}[\text{algorithm finds only isolated vertices}\mid G\in\mathcal{F}_1]=1-o(1)$ and $\mathbb{P}[\text{algorithm finds only isolated vertices}\mid G\in\mathcal{F}_2]=1-o(1)$, hence $\mathbb{P}[\text{algorithm finds only isolated vertices}]=1/2\mathbb{P}[\text{algorithm finds only isolated vertices}|G\in\mathcal{F}_1]+1/2\mathbb{P}[\text{algorithm finds only isolated vertices}|G\in\mathcal{F}_2]=1-o(1)$. The conclusion follows by Bayes's formula.}.
Therefore, the algorithm cannot distinguish between $r$-graphs in $\mathcal{F}_1$ and $\mathcal{F}_2$ with sufficiently high probability with only $Q$ queries.
\end{proof}

If $F$ is a non-$r$-partite $r$-graph, then $\mathit{ex}(n,F)=\Theta(n^{r})$.
Using this, \cref{teor:bound2} asserts that, for any non-$r$-partite $r$-graph $F$, testing $F$-freeness needs $\Omega(({n^{r-1}}/{d})^{{1}/{r}})$ queries.
This implies that for all non-$r$-partite $r$-graphs $F$ there is no constant time $F$-freeness tester for input $r$-graphs $G$ on $n$ vertices with $d=o(n^{r-1})$ and $d=\Omega(1)$, as opposed to the constant time algorithms existing for dense $r$-graphs.

In more generality, \cref{teor:bound2} shows that there can be no $F$-freeness tester that requires a constant number of queries whenever the input $r$-graph $G$ has average degree $d=o(\mathit{ex}(n,F)/n)$ and $d=\Omega(1)$.
On the other hand, if the number of edges of the input $r$-graph is larger than the Tur\'an number of $F$, then there is a trivial $F$-freeness tester: an algorithm that rejects every input, which has constant\COMMENT{zero} query complexity.
As another example, it is well-known that $ex(n,C_4)=\Theta(n^{3/2})$.
With this, we conclude that any algorithm testing $C_4$-freeness in graphs with average degree $d$, when $d=o(n^{1/2})$ and $d=\Omega(1)$, must perform at least $\Omega(({n}/{d^2})^{1/3})$ queries.

The asymptotic growth of $\mathit{ex}(n,F)$ is not known for every $F$.
Let $\beta(F)\coloneqq\frac{v_F-r}{e_F-1}$.
An easy probabilistic argument\COMMENT{Idea of the proof: consider a random $r$-graph (Erdos-Renyi model) with probability $p$. Compute the expected number of copies of $F$. Set $p$ so that this expected number of copies is half the expected number of edges. Delete one edge from each copy; the number of edges is still $\Theta(np)$ and it is now $F$-free.} shows that $\mathit{ex}(n,F)=\Omega\left(n^{r-\beta(F)}\right)$.
This bound is superlinear in $n$ as long as $\beta(F)<r-1$, which holds for every connected $F$ that is not a weak tree\COMMENT{For a weak tree $F$ we have that $v_F=(r-1)e_F+1$. For any other connected $r$-graph $F$ we have $v_F<(r-1)e_F+1$ and the claim follows by reordering.}.
Using this bound on $\mathit{ex}(n,F)$, \cref{teor:bound2} asserts that for any connected $r$-graph $F$ other than a weak tree the number of queries performed by any $F$-freeness tester on input $r$-graphs on at least $\Omega(n)$ and at most $o\left(n^{r-\beta(F)}\right)$ edges is $\Omega(({n^{r-1-\beta(F)}}/{d})^{{1}/{(r-\beta(F))}})$.


\subsection{A lower bound for denser $r$-graphs}\label{section42}

The lower bound on the query complexity of $F$-freeness testers we present here improves the bound in \cref{section41} when $d$ is large enough and either $r=2$ or $r\geq3$ and $F$ is non-$r$-partite.
However, this approach only works for one-sided error algorithms.
The answer given by one-sided error algorithms must always be correct when the input \mbox{$r$-graph} is $F$-free, so any algorithm we consider must accept if it cannot rule out the possibility of $G$ being $F$-free.
Thus, in order to prove that the query complexity is at least $Q$, say, (roughly speaking) the idea is to find a family $\mathcal{F}$ of $r$-graphs which are far from being $F$-free and such that any algorithm, given an $r$-graph chosen uniformly at random from $\mathcal{F}$ as an input, must perform at least $Q$ queries in order to find a copy of $F$ (with high probability).
As we will prove, the family $\mathcal{F}_{n,d(n)}^{(r)}$ described below has the required properties.

Let $F$ be an $r$-graph other than a weak forest.
Recall that $X_F(G)$ denotes the number of copies of $F$ in $G$.
Let $\Phi_{F,n,d}\coloneqq\min\{\mathbb{E}[X_K(G_{n,d}^{(r)})]:K\subseteq F, e_K>0\}$.
Taking $K$ to be an edge shows that $\Phi_{F,n,d'}\leq nd'/r$ for any $d'$.

Assume now that $d(n)=\omega(1)$ and $d(n)=o(n^{r-1})$. Choose $\eta(n)$ such that $\eta(n)=o(1)$.
Let
\[n_*\coloneqq \max\{n_0\leq n:\Phi_{F,n_0,d(n)}\ge (1-\eta(n))n_0d(n)/r\}.\]
We claim that $n_*$ always exists.
Indeed, let $n_1\le n$ be such that there exists an $r$-graph~$G^*$ on $n_1$ vertices with average degree $d(n)$ and at least $(1-\eta(n)^2)\binom{n_1}{r}$ edges.%
\COMMENT{To see that we can choose $n_1\le n$, use that $d=o(n^{r-1})$.}
Thus, $n_1=(1\pm o(1))((r-1)!d(n))^{1/(r-1)}$ and, since $d=\omega(1)$, we have $n_1=\omega(1)$.
Consider any $G^*$ as above.
Given any $K\subseteq F$, note that the number of copies of $K$ in $G^*$ is given by $(1\pm \eta(n))\binom{n_1}{v_K}\frac{v_K!}{\mathrm{aut}(K)}$.%
\COMMENT{$G^*$ is almost complete, and each missing edge destroys at most $n_1^{v_K-r}$ copies of $K$. Thus the number of copies of $K$ in $G^*$ is given by $\binom{n_1}{v_K}\frac{v_K!}{\mathrm{aut}(K)}\pm \eta^2\binom{n_1}{r}n_1^{v_K-r}=(1\pm \eta)\binom{n_1}{v_K}\frac{v_K!}{\mathrm{aut}(K)}$.}
(This can be seen by observing that $G^*$ is ``almost complete'', and that every edge that is removed from a complete $r$-graph on $n_1$ vertices affects at most $n_1^{v_K-r}$ copies of $K$; since only $\eta(n)^2\binom{n_1}{r}$ edges are removed, this gives a total of at most $\eta(n)^2\binom{n_1}{r}n_1^{v_K-r}=o(\eta(n)\binom{n_1}{v_K})$ copies of $K$ affected by the missing edges.)
Among all $K\subseteq F$ with $e_K\geq1$, this expression achieves its minimum (if $n$ is sufficiently large) for a single edge.
Hence $\Phi_{F,n_1,d(n)}\ge (1-\eta(n)) n_1d(n)/r$%
\COMMENT{get $\Phi_{F,n_1,d(n)}\ge (1- \eta)\binom{n_1}{r}\ge  (1- \eta)|G^*|=(1-\eta) n_1d(n)/r$}
and $n_*\geq n_1$ must exist\footnote{Note that here we are using the fact that there exist very dense $d(n)$-regular $r$-graphs. This follows from \cref{remark:intro} by considering the complement.}.

\begin{lemma}\label{lema:cosanueva2}
Let\/ $F$ be a fixed\/ $r$-graph other than a weak forest and let\/ $d(n)$ be such that\/ $d(n)=\omega(1)$ and\/ $d(n)=o(n^{r-1})$.
Then\/ $d(n)=o(n_*^{r-1})$.
\end{lemma}

\begin{proof}
For any fixed $r$-graph $K$ with $e_K>1$, let $d^*(n,K)$ be the smallest integer such that $\mathbb{E}[X_K(G_{n,d^*(n,K)}^{(r)})]\geq nd^*(n,K)/r$.%
\COMMENT{Let $d'=\binom{n-1}{r-1}$. Then $G^{(r)}_{n,d'}=K^{(r)}_n$ and so
$\mathbb{E}[X_K(G_{n,d'})]=\binom{n}{v_K}\frac{v_K!}{\mathrm{aut}(K)}\ge nd'/r$.}
Let $d_F^*(n)\coloneqq\max_{K\subseteq F:e_K>1}\{d^*(n,K)\}$.%
\COMMENT{Note that some $K$ exists as $e_F \ge 2$ as $F$ is not a weak forest.}
We claim that $d_F^*(n)=o(n^{r-1})$.
To prove the claim, note that, by \cref{coro:expectation}\ref{item3}, for any $K\subseteq F$ with $e_K>1$ we have that \[d^*(n,K)=\Theta\left(\left\lceil n^{\frac{(r-1)e_K-v_K+1}{e_K-1}}\right\rceil\right)\COMMENT{$n^{v_K} (d/n^{r-1})^{e_K}=nd$ iff $d^{e_K-1}=n^{-(v_K-1-e_K(r-1))}$.
Note that the exponent could be $<0$, but is well-defined since $K$ is not a single edge.
If the exponent is $<0$, then $d^*(n,K)=\Theta(1)$ as $d^*(n,K)\in\mathbb{N}$ by definition.}.\]%
In particular,  $d^*(n,K)=o(n^{r-1})$ as $v_K>r$.
The claim follows by taking the maximum over all~$K$.

Returning to the main proof, we now consider two cases.
If $n_*=n$, then $d(n)=o(n_*^{r-1})$ by assumption.
So suppose $n_* < n$.
Let $n_+> n_*$ be the smallest integer such that there exists a $d(n)$-regular $r$-graph on $n_+$ vertices.
So $n_+\le 2n_*$\COMMENT{take a disjoint union of two $d(n)$-regular $r$-graphs on $n_*$ vertices each} (since a $d(n)$-regular $r$-graph on $2n_*$ vertices can be constructed by duplicating one on $n_*$ vertices) and $n_+ \le n$\COMMENT{as we only consider $d(n)$ if there exist a $d(n)$-regular $r$-graph on $n$ vertices (though that's not quite explicit in the lemma...)} (because $d(n)=o(n^{r-1})$, see \cref{remark:intro}).
By the definition of $n_*$, $\Phi_{F,n_+,d(n)}<(1-\eta(n))n_+d(n)/r$.
In particular, there exists $K\subseteq F$ with $e_K\geq2$\COMMENT{For one single edge, the expected number of copies is exactly the number of edges, and this is not possible in this range.} such that $\mathbb{E}[X_K(G_{n_+,d(n)}^{(r)})]<(1-\eta(n))n_+d(n)/r$.
By the definition of $d^*(n,K)$ and \cref{coro:expectation}\ref{item3}, we then have that $d(n)<2d^*(n_+,K)$.
This in turn implies that $d(n)<2d^*_F (n_+)$.
But $d_F^*(n_+)=o(n_+^{r-1})$ by the above claim, and thus $d(n)=o(n_*^{r-1})$.
\end{proof}

Let $t\coloneqq\lfloor n/n_*\rfloor$.
Define $\mathcal{F}_{n,d(n)}^{(r)}$ by considering all possible partitions of $V$ into sets $V_1,\ldots,V_t$ of size 
\begin{equation}\label{equa:ntildedef}
\tilde n\coloneqq n/t
\end{equation}
and, for each of them, all possible labelled $d(n)$-regular $r$-graphs $G_i$ on each of the sets $V_i$\COMMENT{Note that $n_*\leq\tilde n\leq2n_*$, so the asymptotic ($\Theta$) behaviour is unaffected by this. Formally, we need to choose the $V_i$ to only have almost equal size in order to ensure that there exists a $d(n)$-regular $r$-graph with vertex set $V_i$.}.
By \cref{lema:cosanueva2}, $d(n)=o(\tilde n^{r-1})$ and so the $G_i$ are well-defined (see \cref{remark:intro}).
With these definitions, all the results in \cref{section2,section31} can be applied to each family $\mathcal{F}_{n,d(n)}^{(r)}[V_i]$ consisting of the subgraphs of each $G\in\mathcal{F}_{n,d(n)}^{(r)}$ restricted to vertex set $V_i$, and hence to $\mathcal{F}_{n,d(n)}^{(r)}$ by summing over all $i\in[t]$.

\begin{lemma}\label{lema:distrib}
Let\/ $F$ be a fixed, connected\/ $r$-graph other than a weak tree and let\/ $d(n)$ be such that\/ $d(n)=\omega(1)$ and\/ $d(n)=o(n^{r-1})$.
Let\/ $\tilde n$ and\/ $\mathcal{F}_{n,d(n)}^{(r)}$ be as defined above.
Then, an\/ $r$-graph\/ $G\in\mathcal{F}_{n,d(n)}^{(r)}$ chosen uniformly at random contains\/ $\Theta(nd(n))$ edge-disjoint copies of\/ $F$ a.a.s.
\end{lemma}

Note that this immediately implies that a.a.s.~a graph $G\in\mathcal{F}_{n,d(n)}^{(r)}$ chosen uniformly at random is $\varepsilon$-far from being $F$-free for some fixed $\varepsilon>0$.

\begin{proof}
Let $D_F(G)$ denote the maximum number of edge-disjoint copies of $F$ in an $r$-graph $G$.
Recall that $\mathcal{F}_{n,d(n)}^{(r)}$ is obtained by partitioning the set of vertices into sets $V_1,\ldots,V_t$ of size $\tilde{n}$, where $t=n/\tilde{n}$, and considering $d(n)$-regular $r$-graphs $G_i$ on each of the $V_i$, where each $G_i$ is chosen uniformly at random from $\mathcal{G}_{\tilde{n},d(n)}^{(r)}$, independently of each other.
Note that  $n_*\leq\tilde n\leq2n_*$. 
Together with the definition of $n_*$ and \cref{coro:expectation}\ref{item3}, this implies that the value of $\Phi_{F,\tilde n,d(n)}$ in each $G_i$ satisfies $\Phi_{F,\tilde n,d(n)}=\Theta(\tilde nd(n))$.
Then, by \cref{lema:disjointcopies}, for any fixed $i\in[t]$, the maximum number of edge-disjoint copies of $F$ in $G_i$ is $D_F(G_i)=\Theta(\tilde nd(n))$ a.a.s.

We now claim that a graph $G\in\mathcal{F}_{n,d(n)}^{(r)}$ chosen uniformly at random a.a.s.~satisfies that $D_F(G)=\Theta(nd(n))$.
Observe that the bound $D_F(G)=\bigO(nd(n))$ is trivial, as $G$ has exactly $nd(n)/r$ edges.
For the lower bound, since $D_F(G_i)=\Theta(\tilde nd(n))$ a.a.s.~for each $i\in[t]$, by the independence of the choice of $G_i$ we have that a.a.s.~at least half of the graphs $G_i$ satisfy this equality.
Therefore, $D_F(G)=\Omega(t\tilde nd(n))=\Omega(nd(n))$.
\end{proof}

We now provide a proof for the lower bound on the complexity of any algorithm that tests $F$-freeness in $r$-graphs (for graphs and non-$r$-partite $r$-graphs $F$ with $r\geq3$).
In order to do so, consider any algorithm ALG that performs $Q$ queries given an input $r$-graph $G$ on $n$ vertices with average degree $d(n)\pm o(d(n))$.
ALG will retrieve some information about $G$ from the queries it performs, namely a set of $r$-sets $E_1\subseteq E(G)$, a set of $r$-sets $E_2\subseteq E(\overline{G})$ and (potentially) some vertex degrees of $G$, i.e.~a set $\mathcal{D}\subseteq\{(v,d_v):v\in V(G), d_v=\deg_G(v)\}$.
We call the information retrieved by ALG after $Q$ queries the \textit{history} of $G$ seen by ALG, and denote it as $(E_1,E_2,\mathcal{D})$.
We say that the history of $G$ seen by ALG is \textit{simple} if $E_1$ forms a weak forest and for all $(v,d_v)\in\mathcal{D}$ we have that $d_v=\bigO(d(n))$.

We will allow our algorithm to find weak forests in the input graphs.
Thus we assume that $F$ is not a weak forest, that is, $F$ contains at least two edges whose intersection has size at least $2$ or a loose cycle.
In order to prove our bound we first show the following result.

\begin{lemma}\label{teor:secondbound2}
Let\/ $F$ be an\/ $r$-graph which is not a weak forest and define\/ $\tilde n$ as in \eqref{equa:ntildedef}.
Assume that\/ $d(n)=\omega(1)$ and\/ $d(n)=o(n^{r-1})$.
Suppose ALG is an algorithm whose input is an\/ $r$-graph\/ $G\in\mathcal{F}_{n,d(n)}^{(r)}$ and which for at least\/ $1/3$ of the\/ $r$-graphs\/ $G\in\mathcal{F}_{n,d(n)}^{(r)}$ sees with probability at least\/ $1/3$ a history which is not simple.
Then, ALG must perform\/ $\Omega(\min\{d(n),\tilde{n}^{r-1}/d(n),\tilde{n}^{1/2}\})$ queries.
\end{lemma}

To prove \cref{teor:secondbound2}, we will show that an algorithm that performs only $o(\min\{d(n),\tilde{n}^{r-1}/d(n),\tilde{n}^{1/2}\})$ queries will usually not succeed with the desired probability.
For this, we consider a suitable randomised process $P$ that answers the queries of the algorithm.

\begin{proof}
Suppose $Q=o(\min\{d(n),\tilde{n}^{r-1}/d(n),\tilde{n}^{1/2}\})$.
Let ALG be a (possibly adaptive and randomised) algorithm that performs $Q$ queries and searches for some history of the input $G\in\mathcal{F}_{n,d(n)}^{(r)}$ which is not simple.
Since we have, for any history $(E_1,E_2,\mathcal{D})$ seen by any algorithm on any $G\in\mathcal{F}_{n,d(n)}^{(r)}$, that any pair $(v,d_v)\in\mathcal{D}$ satisfies $d_v=d(n)$, the only condition for $(E_1,E_2,\mathcal{D})$ being simple is that $E_1$ forms a weak forest.
Therefore, ALG tries to find a set $E_1\subseteq E(G)$ which forms an $r$-graph which is not a weak forest.

The queries performed by ALG are answered by a randomised process $P$.
We denote the queries asked by ALG as $q_1,q_2,\ldots$, and the answers given by $P$ as $a_1,a_2,\ldots$.
After $t$ queries, we refer to all the previous queries from ALG and all the answers provided by $P$ as the \emph{query-answer history}.
The process $P$ uses the query-answer history to build what we call the \emph{history book}, defined for each $t\geq0$ and denoted by $H^t=(V^t,E_*^t,\bar{E}^t)$, where $V^t\subseteq V$, $\bar{E}^t\subseteq\binom{V}{r}$ and $E_*^t$ is a set of labelled $r$-sets in $\binom{V}{r}$ such that each $r$-set $e\in E_*^t$ has $r$ labels $i_1,\ldots,i_r$, one for each vertex in $e$.
We denote by $E^t$ the set of edges consisting of the $r$-sets in $E^t_*$.
Given an edge $e=\{v_1,\ldots,v_r\}\in E^t$, its labels in $E^t_*$ indicate, for each vertex $v_j\in e$, that $e$ is the $i_j$-th edge in the incidence list of $v_j$.

Initially, $V^0$, $E_*^0$ and $\bar{E}^0$ are set to be empty.
Note that we may always assume that in the $t$-th step ALG never asks a query whose answer can be deduced from the history book $H^{t-1}$.
Given two $r$-graphs $H$ and $H'$, define $\mathcal{F}_{n,d(n),H,H'}^{(r)}\coloneqq\{G\in\mathcal{F}_{n,d(n)}^{(r)}:H\subseteq G, H'\subseteq\overline{G}\}$.
We abuse notation to write $\mathcal{F}_{n,d(n),H,H'}^{(r)}$ as the event that $G\in\mathcal{F}_{n,d(n),H,H'}^{(r)}$.
The process $P$ answers ALG's queries and builds the history book as follows.

If $q_t=\{v_1,\ldots,v_r\}$ is a vertex-set query, then $P$ answers ``yes'' with probability $\mathbb{P}[q_t\in G\mid\mathcal{F}_{n,d(n),E^{t-1},\bar{E}^{t-1}}^{(r)}]$, and ``no'' otherwise.
If the answer is ``yes'', then the history book is updated by setting $V^{t}\coloneqq V^{t-1}\cup q_t$, $\bar{E}^{t}\coloneqq\bar{E}^{t-1}$ and adding $q_t$ together with its labels $j_1,\ldots,j_r$ to $E_*^{t-1}$ to obtain $E_*^{t}$, where the labels $j_1,\ldots,j_r$ are chosen uniformly at random among all possible labellings which are consistent with the labels in $E_*^{t-1}$.
In this case, the labels are also given to ALG as part of the answer.
Otherwise, the history book is updated by setting $V^{t}\coloneqq V^{t-1}\cup q_t$, $E_*^{t}\coloneqq E_*^{t-1}$ and $\bar{E}^{t}\coloneqq\bar{E}^{t-1}\cup\{q_t\}$.

If $q_t=(u,i)$ is a neighbour query, $P$ replies with $a_t\coloneqq(v_1,\ldots,v_{r-1},j_1,\ldots,j_{r-1})$, where $a_t$ is chosen such that $e\coloneqq\{u,v_1,\ldots,v_{r-1}\}$ is an edge and for each $k\in[r-1]$, the number $j_k$ is the position of $e$ in the incidence list of $v_k$ (we may assume that, as the $r$-graphs are $d(n)$-regular, the algorithm never queries $i>d(n)$).
To determine its answer $a_t$, the process $P$ will first choose an $r$-graph $G_t\in\mathcal{F}_{n,d(n),E^{t-1},\bar{E}^{t-1}}^{(r)}$ uniformly at random, and then choose a labelling of the edges of $G_t$ which is consistent with $H^{t-1}$ uniformly at random.
The edge $e=\{u,v_1,\ldots,v_{r-1}\}$ will be the $i$-th edge at $u$ in $G_t$ (in the chosen labelling) and $j_s$ will be the label of $e$ in the incidence list of $v_s$ (for each $s\in[r-1]$).
Note that the random labelling ensures that, given $G_t$, $e$ is chosen uniformly at random from a set of edges of size at least $d(n)-t$ (namely from the set of those edges of $G_t$ incident to $u$ which have no label at $u$ in $H^{t-1}$).
This in turn means that for all $f\in G_t$ with $u\in f$, the probability that the label of $u$ in $f$ is $i$ is at most $1/(d(n)-t)$.
The history book is updated by setting $V^{t}\coloneqq V^{t-1}\cup e$, $\bar{E}^{t}\coloneqq\bar{E}^{t-1}$ and adding $e$ together with the labels $i,j_1,\ldots,j_{r-1}$ to $E_*^{t-1}$ to obtain $E_*^{t}$.

Once $P$ has answered all $Q$ queries, it chooses an $r$-graph $G^*\in\mathcal{F}_{n,d(n),E^Q,\bar{E}^Q}^{(r)}$ uniformly at random.
Note that $P$ gives extra information to the algorithm in the form of labels that have not been queried.
This extra information can only benefit the algorithm, so any lower bound on the query complexity in this setting will also be a lower bound in the general setting\COMMENT{We give this extra information so that the algorithm never queries the same edge twice, which may happen with different types of queries otherwise.}.

We claim that $G^*$ is chosen uniformly at random in $\mathcal{F}_{n,d(n)}^{(r)}$.
Indeed, let $s_0\coloneqq|\mathcal{F}_{n,d(n)}^{(r)}|$.
Given a query-answer history $\mathcal{H}=(q_1,a_1,\dots,q_Q,a_Q)$, for each $t\in [Q]\cup \{0\}$, write $\mathcal{F}^t(\mathcal{H})$ for the set of all those graphs $G\in \mathcal{F}_{n,d(n)}^{(r)}$ which are ``consistent'' with $\mathcal{H}$ for at least the first $t$ steps, i.e.~all $G\in \mathcal{F}_{n,d(n),E^t,\bar{E}^t}^{(r)}$, where $(V^t, E^t_*,\bar{E}^t)$ is the history book associated with the first $t$ steps of~$\mathcal{H}$.
Thus $\mathcal{F}^t$ is a random variable and $\mathcal{F}^0(\mathcal{H})=\mathcal{F}_{n,d(n)}^{(r)}$ for each $\mathcal{H}$.
Now consider any sequence $\mathcal{S}=(s_1,\ldots,s_Q)$ such that $s_t\in\mathbb{N}$ and $\mathbb{P}[|\mathcal{F}^t|=s_t]>0$ for all $t\in[Q]$.
Write $\mathbb{P}_{\mathcal{S}}$ for the probability space consisting of all those query-answer histories $\mathcal{H}=(q_1,a_1,\dots,q_Q,a_Q)$ which satisfy $|\mathcal{F}^t|=s_t$ for all $t\in[Q]$.
Take any fixed $r$-graph $G\in\mathcal{F}_{n,d(n)}^{(r)}$.
Note that our choice of the $t$-th answer $a_t$ given by $P$ implies that $\mathbb{P}_{\mathcal{S}}\left[G\in\mathcal{F}^t\mid G\in \mathcal{F}^{t-1}\right]=s_t/s_{t-1}$ for all $t\in [Q]$.%
\COMMENT{Indeed, consider any query-answer history 
$\mathcal{H}^{t-1}=(q_1,a_1,\dots,q_{t-1},a_{t-1})$ up to step $t-1$. Below, we are working in the probability space  $\mathbb{P}_{\mathcal{S}}$, so we only consider $\mathcal{H}^{t-1}$ and $\mathcal{H}^{t}$ with  $|\mathcal{F}^{t-1}(\mathcal{H}^{t-1})|=s_t$ and  $|\mathcal{F}^t(\mathcal{H}^{t})|=s_t$, respectively. Let 
$(V^{t-1}, E^{t-1}_*,\bar{E}^{t-1})$ be the history book associated with $\mathcal{H}^{t-1}$. Suppose that $G$ is consistent with $\mathcal{H}^{t-1}$, i.e.~$G\in \mathcal{F}_{n,d(n),E^{t-1},\bar{E}^{t-1}}^{(r)}$.\\
Let us first consider the case when $q_t$ is a vertex-set query.
Then $a_t$ is ``yes'' with probability $|\mathcal{F}_{n,d(n),E^{t-1}\cup q_t,\bar{E}^{t-1}}^{(r)}|/s_{t-1}$ and ``no'' with probability $1-|\mathcal{F}_{n,d(n),E^{t-1}\cup q_t,\bar{E}^{t-1}}^{(r)}|/s_{t-1}=|\mathcal{F}_{n,d(n),E^{t-1},\bar{E}^{t-1}\cup q_t}^{(r)}|/s_{t-1}$. Note that if $a_t$ is ``yes'' then $|\mathcal{F}_{n,d(n),E^{t-1}\cup q_t,\bar{E}^{t-1}}^{(r)}|=s_t$, while if $a_t$ is ``no'' then $|\mathcal{F}_{n,d(n),E^{t-1},\bar{E}^{t-1}\cup q_t}^{(r)}|=s_t$. Now, if $q_t\in G$, then
$G$ is consistent with $\mathcal{H}^t=(q_1,\dots,q_{t},a_1,\dots,a_{t})$ iff $a_t$ is ``yes'', which happens with probability $s_t/s_{t-1}$, while if  $q_t\notin G$, then $G$ is consistent with $\mathcal{H}^t$ iff $a_t$ is ``no'', which also happens with probability $s_t/s_{t-1}$.\\
Now suppose that $q_t=(u,i)$ is a neighbour query. Note that each $G'$ which is consistent with $\mathcal{H}^{t-1}$ has the same number of ``free'' (i.e.~yet unlabelled/unseen) edges at $u$. Let $N$ be this number. Recall that the answer $a_t$ given by $P$ is some edge $e$ at $u$, where $e$ is chosen by first choosing a graph $G_t$ which is consistent with $\mathcal{H}^{t-1}$ uniformly at random, and then choosing one of the $N$ ``free'' edges at $u$ uniformly at random. Now, $G$ is consistent with $\mathcal{H}^t$ iff $e$ is one of the free edges in $G$ at $u$. Let $e_1,\dots,e_N$ be these free edges. For each $j$ we have $e_j=e$ iff $e_j\in G_t$ and $e_j$ is chosen to be the $i$-th edge at $u$ (the latter happens with probability $1/N$). Also
 $\mathbb{P}_{\mathcal{S}}[e_j\in G_t]=|\mathcal{F}_{n,d(n),E^{t-1}\cup e_j,\bar{E}^{t-1}}^{(r)}| /s_{t-1}$, but if $e=e_j$, then
$|\mathcal{F}_{n,d(n),E^{t-1}\cup e_j,\bar{E}^{t-1}}^{(r)}|=s_t$. 
Thus  $\mathbb{P}_{\mathcal{S}}[e_j=e]=s_t / (Ns_{t-1})$. As this holds for each $j$, and the events $e_j=e$ are disjoint, it follows that $\mathbb{P}_{\mathcal{S}}[e\in G]=s_t / s_{t-1}$.
}
Thus,
\begin{align*}
\mathbb{P}_{\mathcal{S}}[G^*=G]=&\,\mathbb{P}_{\mathcal{S}}\left[G\in\mathcal{F}^Q\right]/s_Q
=\,\frac{1}{s_Q}\left(\prod_{t=1}^{Q}\mathbb{P}_{\mathcal{S}}\left[G\in\mathcal{F}^t\mid G\in \mathcal{F}^{t-1}\right]\right)\mathbb{P}_{\mathcal{S}}\left[G\in\mathcal{F}^0\right]\\
=&\,\frac{1}{s_Q}\prod_{t=1}^{Q}\frac{s_t}{s_{t-1}}=\frac{1}{s_0}.
\end{align*}
Thus $\mathbb{P}[G^*=G]=1/|\mathcal{F}_{n,d(n)}^{(r)}|$ by the law of total probability.

Now let us prove that ALG will a.a.s.~only see a simple history $(E_1,E_2,\mathcal{D})$.
Note that $E_1=E^Q$ and $E_2=\bar{E}^Q$.
Hence it suffices to show that $E^Q$ is a weak forest a.a.s.
Recall that we can write each $G\in\mathcal{F}_{n,d(n)}^{(r)}$ as the disjoint union of $G_1,\ldots,G_s$, where $s=n/\tilde n$, each $G_j$ is uniformly distributed in $\mathcal{G}_{\tilde n,d(n)}^{(r)}$ and $G_j$ has vertex set $V_j$.

Assume $q_t$ is a vertex-set query.
The probability that $P$ answers ``yes'' is given by \cref{coro:switchprob} as $\bigO(d(n)/\tilde{n}^{r-1})$, as long as $t=o(d(n))$\COMMENT{If $t=o(d(n))$, the number of edges in $E$, $\bar{E}$ is at most $o(d(n))$, hence the conditions for the statement of the corollary hold. Note that we are applying \cref{coro:switchprob} on some of the $G_j$ (we can do this because $d=o(\tilde n^{r-1})$ by \cref{lema:cosanueva2} and since $n_*\leq\tilde n\leq2n_*$); the reason why we have $\bigO$ and not $\Theta$ is that there is the possibility that vertices in the query lie in distinct $V_j$, in which case the answer is always ``no''.}.
Thus, because the number of queries is $Q=o(\tilde{n}^{r-1}/d(n))$, then, by a union bound, the probability that any edge is found with vertex-set queries is $o(1)$.

Assume now that $q_t=(u,i)$ is a neighbour query, where $u\in V_j$, $j\in[s]$.
We will bound the probability that some vertex returned by $P$ in the $t$-th answer $a_t$ lies in $V^{t-1}$\COMMENT{If there is no such vertex, no loose cycle or two edges sharing two vertices are found, so $E^t$ will still be a weak forest.}.
Note that such a vertex will always lie in $V_j$.
To bound this probability, for any given vertex $v\in V^{t-1}\cap V_j$, let $S_v\coloneqq\{f\in\binom{V_j}{r}:u,v\in f, f\notin E^{t-1}\cup \bar{E}^{t-1}\}$.
Note that $|S_v|=\bigO(\tilde{n}^{r-2})$.
By \cref{coro:switchprob}, the probability that a given $r$-set in $S_v$ is an edge of $G_t$ is $\Theta(d(n)/\tilde{n}^{r-1})$.
Furthermore, if we condition on $e\in S_v$ being an edge of $G_t$, recall that the probability that its label belonging to $u$ equals $i$ is at most $1/(d(n)-t)$.
Thus, by a union bound over all elements of $S_v$, the probability that some $r$-set in $S_v$ is the $i$-th edge in the incidence list of $u$ is $\bigO(d(n)/(\tilde{n}(d(n)-t)))=\bigO(1/\tilde n)$.
Note that $|V^{t-1}|\leq rt=\bigO(Q)$.
Thus, by a union bound over all $v\in V^{t-1}$, the probability that the answer to the $t$-th query results in $E^{t}$ being not a weak forest is $\bigO(Q/\tilde{n})$.
By a union bound over all queries, the probability that any of the at most $Q$ neighbour queries finds any vertex in the current history book is $\bigO\left(Q^2/\tilde{n}\right)=o(1)$.
This in turn implies that the probability that a neighbour query detects anything else than a weak forest is $o(1)$.

Combining the conditions and lower bounds for both types of queries, we have that the probability (taken over all queries of ALG and choices of $P$ in the above process) that $E^Q$ is not a weak forest is $o(1)$.
The statement follows since we have shown that the $r$-graph $G^*$ returned by $P$ is chosen uniformly at random from $\mathcal{F}_{n,d(n)}^{(r)}$\COMMENT{Indeed, suppose the proportion of inputs where ALG succeeds is at least $1/3$ (here ``succeeds'' means that ALG sees a history that is not simple with probability at least $1/3$). This is a contradiction.}.
\end{proof}

\begin{theorem}\label{teor:secondbound}
The following statements hold:
\begin{enumerate}[label=(\roman*)]
\item Let\/ $F$ be a connected graph which is not a tree.
Assume that\/ $d(n)=\omega(1)$ and\/ $d(n)=o(n)$.
Assume, furthermore, that\/ $nd(n)/2\leq\mathit{ex}(n,F)$\COMMENT{If $nd(n)/2>(1+o(1))\mathit{ex}(n,F)$, the algorithm always says the graph is far from $F$-free. This is a one-sided error tester with constant query complexity (in general, the algorithm may do this if the number of edges is part of the input and is above the extremal number).}.
Then, any one-sided error\/ $F$-freeness tester must perform\/ $\Omega(\min\{d(n),\tilde{n}/d(n),\tilde{n}^{1/2}\})$ queries when restricted to\/ $n$-vertex inputs of average degree\/ $d(n)-o(d(n))$\COMMENT{We don't want $d(n)\pm o(d(n))$ here since we need $|G|\leq\mathit{ex}(n,F)$ for our input graphs $G$.}, where\/ $\tilde n$ is as defined in \eqref{equa:ntildedef}.\label{teor:finalthm1}
\item Let\/ $r\geq3$.
Let\/ $F$ be a connected non-$r$-partite\/ $r$-graph.
Assume that\/ $d(n)=\omega(1)$ and\/ $d(n)=o(n^{r-1})$.
Then, any one-sided error\/ $F$-freeness tester in\/ $r$-graphs must perform\/ $\Omega(\min\{d(n),\tilde{n}^{r-1}/d(n),\tilde{n}^{1/2}\})$ queries when restricted to\/ $n$-vertex inputs of average degree\/ $d(n)\pm o(d(n))$, where\/ $\tilde n$ is as defined in \eqref{equa:ntildedef}.\label{teor:finalthm2}
\end{enumerate}
\end{theorem}

\begin{proof}
We first prove \ref{teor:finalthm2} and later discuss which modifications are needed to prove~\ref{teor:finalthm1}.
Let $Q=o(\min\{d(n),\tilde{n}^{r-1}/d(n),\tilde{n}^{1/2}\})$.
Consider any algorithm ALG that performs $Q$ queries given an input $r$-graph $G$ on $n$ vertices with average degree $d(n)\pm o(d(n))$.
Assume that ALG is given an $r$-graph $G\in\mathcal{F}_{n,d(n)}^{(r)}$ as an input.
By \cref{teor:secondbound2} we know that any algorithm that performs at most $Q$ queries will see a simple history $(E_1,E_2,\mathcal{D})$ of $G$ with probability at least $2/3$ for at least $2/3$ of the graphs $G\in\mathcal{F}_{n,d(n)}^{(r)}$\COMMENT{Recall that the condition on the degrees holds trivially.}.
Note that any such simple history $(E_1,E_2,\mathcal{D})$ is such that $|E_1\cup E_2|\leq Q$\COMMENT{as the algorithm performs $Q$ queries}, $|\mathcal{D}|\leq Q$\COMMENT{In fact, $\leq rQ/(r+1)$ if the algorithm receives extra information as indices, and $\leq Q/2$ otherwise. In the former case, the best case scenario is given by performing one vertex-set query and then determining each degree of the $r$ verties with one neighbour query each.} and for every $(v,d_v)\in\mathcal{D}$, $d_v=d(n)$\COMMENT{As the $r$-graphs we are considering here are $d(n)$-regular}.
We now show that there is a family $\mathcal{F}_2$ of $F$-free $r$-graphs that, for every such simple history, contains at least one $r$-graph for which ALG will see the same history with positive probability.
\begin{itemize}
\item For each simple history $(E_1,E_2,\mathcal{D})$, let $H$ be the $r$-graph that has vertex set $\bigcup_{e\in E_1\cup E_2}e$ and edge set $E_1$.
Note that $H$ is a weak forest with (possibly) some isolated vertices and $|V(H)|\leq rQ$.
Consider a partition of $V(H)$ into $V_1,\ldots,V_r$ such that for every $e\in E(H)$ and $i\in[r]$, we have $|e\cap V_i|=1$, which can be constructed inductively by adding the edges of $H$ one by one and distributing their vertices into different parts.
Consider pairwise disjoint sets of vertices $W_1,\ldots,W_r$ of size $d(n)^{1/(r-1)}$ which are disjoint from $V(H)$.
\item Define an $r$-graph $K$ with vertex set $V(H)\cup W_1\cup\ldots\cup W_r$.
Note that for each $v\in V_i$ there are $d(n)$ $r$-sets $f$ such that $v\in f$ and $|f\cap W_j|=1$ for all $j\in[r]\setminus\{i\}$.
Define $E(K)$ by including $E(H)$ and adding $d(n)-\deg_{H}(v)$ of these $r$-sets incident to each vertex $v\in V(H)$.
Note that $K$ is $r$-partite and, thus, $F$-free.
\item Finally, for any such $K$, consider the $r$-graph $G$ obtained as the vertex-disjoint union of $K$ and any $F$-free $r$-graph on $n-|V(K)|$ vertices with average degree $d(n)-o(d(n))$ (to see that this is possible, note that $|V(K)|=o(n)$\COMMENT{Note that $|V(H)|=\bigO(\tilde n^{1/2})=o(n)$ and $|V(K)|=|V(H)|+rd(n)^{1/(r-1)}=o(n)$ as $d(n)=o(n^{r-1})$.} and $\mathit{ex}(n,F)=\Theta(n^r)$ since $F$ is non-$r$-partite).
\end{itemize}
We define $\mathcal{F}_2$ as the family that consists of all $r$-graphs $G$ that can be constructed as above and all possible relabellings of their vertices.
Note that each $G\in\mathcal{F}_2$ has $n$ vertices, average degree $d(n)\pm o(d(n))$\COMMENT{Since $\deg_K(v)=d(n)$ for all $v\in V(H)$ and every edge of $K$ has at least one vertex in $H$, the number of edges in $K$ is upper bounded by $d(n)|V(H)|=o(nd(n))$. The number of edges in the rest of the $r$-graph is $(n-o(n))(d(n)\pm o(d(n)))/r=(1\pm o(1))nd(n)/r$, hence $|G|=(1\pm o(1))nd(n)/r$.} and is $F$-free.
Moreover, for every $G\in\mathcal{F}_{n,d(n)}^{(r)}$ and any simple history $(E_1,E_2,\mathcal{D})$ seen by ALG on $G$, there is some $r$-graph $G\in\mathcal{F}_2$ such that ALG would have seen $(E_1,E_2,\mathcal{D})$ on $G$\COMMENT{Here we mean that the same run of ALG, with the same coin flips, would have seen the same history.}.

Now suppose ALG is a one-sided error $F$-freeness tester for $r$-graphs of average degree $d(n)\pm o(d(n))$ that performs $Q$ queries.
Assume that ALG is given inputs as follows.
With probability $99/100$, the input is an $r$-graph $G\in\mathcal{F}_{n,d}^{(r)}$ chosen uniformly at random.
With probability $1/100$, the input is an $r$-graph $G\in\mathcal{F}_2$ chosen uniformly at random.
By \cref{teor:secondbound2}, the proportion of $r$-graphs $G\in\mathcal{F}_{n,d(n)}^{(r)}$ for which with probability at least $2/3$ ALG only sees a simple history is least $2/3$.
Moreover, since ALG is a one-sided error tester, it can only reject an input $G$ if ALG can guarantee the existence of a copy of $F$ in $G$.
Thus, if after $Q$ queries ALG has seen a simple history $(E_1,E_2,\mathcal{D})$, then it cannot reject the input, as there are $r$-graphs $G\in\mathcal{F}_2$ which are $F$-free and for which ALG may see the same history with positive probability.
So given a random input as described above, the probability that ALG accepts is at least $(99/100)(2/3)^2>2/5$.

On the other hand, by \cref{lema:distrib}, the proportion of $r$-graphs in $\mathcal{F}_{n,d(n)}^{(r)}$ that are $\varepsilon$-far from being $F$-free is at least $99/100$.
Since ALG is a one-sided error $F$-freeness tester, it must reject these inputs with probability at least $2/3$.
Therefore, given a random input $G$, the probability that ALG rejects $G$ must be at least $(99/100)^2(2/3)>3/5$.
This is a contradiction to the previous statement, so ALG cannot be a one-sided error $F$-freeness tester.

In order to prove \ref{teor:finalthm1}, let $Q=o(\min\{d(n),\tilde{n}/d(n),\tilde{n}^{1/2}\})$.
If $F$ is not bipartite, then \ref{teor:finalthm2} already shows the desired statement.
In order to deal with bipartite graphs $F$, define a new family $\mathcal{F}_1$ (which also works for non-bipartite $F$) as follows.
Given a simple history $(E_1,E_2,\mathcal{D})$, define $H$ as above.
For each $v\in V(H)$, consider $d(n)-\deg_H(v)$ new vertices and add an edge between $v$ and each of them.
Denote the resulting graph by $K$.
Finally, consider the graph $G$ obtained as the disjoint union of $K$ and any $F$-free graph on $n-|V(K)|$ vertices with average degree $d(n)-o(d(n))$\COMMENT{Note that the number of new vertices we attach to $H$ is $\leq dQ=o(\tilde n)=o(n)$ as $Q=o(\tilde n/d(n))$. The average degree in $K$ is clearly upper bounded by $d$. Note that an $F$-free graph with the desired properties exists, as we are assuming that $nd/2\leq\mathit{ex}(n,F)$. The average degree of $G$ is also $d(n)-o(d(n))$.}.
We define $\mathcal{F}_1$ as the family that consists of all graphs $G$ that can be constructed as above and all possible relabellings of their vertices.
The remainder of the proof works in the same way.
\end{proof}

Note that if, for instance, $d(n)=2\mathit{ex}(n,F)/n$ and $F=C_4$, then \cref{teor:secondbound}\ref{teor:finalthm1} (together with \cref{coro:expectation}\ref{item3}) implies a lower bound of $\Omega(n^{1/2})$\COMMENT{We have that $\Phi_{C_4,n_0,d}=\Theta(\min\{n_0d,d^4\})$, so $\tilde n=n$ whenever $d(n)=\Omega(n^{1/3})$ and $\tilde n=d(n)^3$ otherwise. As we know $\mathit{ex}(n,F)=\Theta(n^{3/2})$, this gives $\tilde n=n$ and so a lower bound of $\Omega(n^{1/2})$.}\COMMENT{In general, \cref{teor:secondbound}\ref{teor:finalthm1} gives rise to an extremely sharp threshold in the behaviour of $F$-freeness testers if $F$ is a bipartite graph.
If the input graph is such that the number of edges it has is larger than the Tur\'an number of $F$, then the algorithm may output that the $r$-graph contains copies of $F$ without performing any queries.
However, if the number of edges is at most $\mathit{ex}(n,F)$, the algorithm must perform a polynomial number of queries in order to find any copy of $F$, as shown in \cref{teor:secondbound2}.}.
The bound on the number of queries in \cref{teor:secondbound} is stronger than in \cref{teor:bound2} as long as $d$ is not too small.


\subsection{Upper bounds}\label{section43}

Here, we present several upper bounds on the query complexity for testing $F$-freeness. 
All the testers we present here are one-sided error testers.
Note that there is always the trivial bound of $\bigO(nd)$ queries; the forthcoming results are only relevant whenever the presented bound is smaller than this.
\Cref{teor:upbound1} provides a bound on the query complexity which applies to input $r$-graphs $G$ in which the maximum degree does not differ too much from the average degree.
\Cref{teor:upbound15} improves \cref{teor:upbound1} for special $r$-graphs $F$.
Finally, \cref{teor:upbound3} provides a bound which works for arbitrary $F$ and $G$.
\Cref{teor:upbound1,teor:upbound15} give stronger bounds for very sparse $r$-graphs $G$, whereas \cref{teor:upbound3} gives stronger bounds for denser $r$-graphs.

We will say that a tester for a property $\mathcal{P}$ is an $\varepsilon'$-tester if it is a valid tester for $\mathcal{P}$ for all distance parameters $\varepsilon\geq\varepsilon'$ (recall that $\varepsilon$ stands for the proportion of edges of a graph $G$ that needs to be modified to satisfy a given property $\mathcal{P}$ in order for $G$ to be considered far from $\mathcal{P}$).
The techniques of our algorithms are based on two strategies: random sampling and local exploration.
We will always write $V$ for the vertex set of the input $r$-graph $G$ and $d$ for its average degree.
Given any $S\subseteq V$, we denote by $G[S]\coloneqq\{e\in G:e\subseteq S\}$ the subgraph of $G$ spanned by $S$.
Thus $V(G[S])=S$.
We denote by $G\{S,\rho\}\coloneqq\{e\in G:\exists\ v\in e:\dist(S,v)<\rho\}$ the graph obtained from $G$ by performing a breadth-first search of depth $\rho$ from $S$.
Throughout this section, the hidden constants in the $\bigO$ notation will be independent of both $\varepsilon$ and $n$.
When the constants depend on $\varepsilon$, we will denote this by writing~$\bigO_\varepsilon$.

\begin{proposition}\label{teor:upbound1}
For every\/ $\varepsilon>0$, the following holds.
Let\/ $F$ be a fixed, connected\/ $r$-graph and let\/ $D$ be its diameter.
For the class consisting of all input\/ $r$-graphs\/ $G$ on\/ $n$ vertices with average degree\/ $d$ and maximum degree\/ $\Delta(G)=\bigO(d)$, there exists an $\varepsilon$-tester for\/ $F$-freeness with\/ $\bigO_\varepsilon(d^{D+1})$ queries.\COMMENT{Here I want an upper bound for all $F$. This is only going to be good when $d$ is small enough, but still interesting.}
\end{proposition}

\begin{proof}
We consider a one-sided error $F$-freeness $\varepsilon$-tester.
The procedure is as follows.
First choose a set $S\subseteq V(G)$ of size $\Theta(1/\varepsilon)$ uniformly at random.
For each $v\in S$, find $G\{v,D+1\}$ by performing neighbour queries.
If any of the graphs $G\{v,D+1\}$ contains a copy of $F$, the algorithm rejects $G$.
Otherwise, it accepts it.
Clearly, the complexity is $\bigO(d^{D+1}/\varepsilon)$ and the procedure will always accept $G$ if it is $F$-free.

Assume now that the input is $\varepsilon$-far from being $F$-free.
Then, it contains at least $\varepsilon nd/r$ edges that belong to copies of $F$.
It follows that the number of vertices that belong to some copy of $F$ is $\Omega(\varepsilon nd/\Delta(G))=\Omega(\varepsilon n)$\COMMENT{because the maximum degree of a vertex is $\bigO(d)$}.
Therefore, if the implicit constant in the bound on $|S|$ is large enough, the algorithm will choose one of the vertices that belong to a copy of $F$ with probability at least $2/3$.
If it chooses such a vertex, then, as $F$ has diameter $D$\COMMENT{any copy of $F$ that the vertex belongs to will be contained in $G\{v,D+1\}$.
Once the algorithm finds the copy of $F$, it rejects the input.
As the probability that we find a vertex that belongs to a copy of $F$ is at least $2/3$, the probability that the input is rejected is also at least $2/3$.}, it rejects the input.
\end{proof}

We can improve the bound in \cref{teor:upbound1} for a certain class of $r$-graphs $F$.
Given any $r$-graph $F$, let $D_F$ be its diameter.
Consider the partition of its vertices given by choosing an edge $e\in F$, taking $V_0(e)\coloneqq e$ and $V_i(e)\coloneqq\{u\in V(F):\mathrm{dist}(e,u)=i\}$ for $i\in[D_F]$.
We let $\mathcal{F}_E\coloneqq\{F:|F[V_{D_F}(e)]|=0\ \ \forall\ e\in F\}$.
The class $\mathcal{F}_E$ contains, for instance, complete $r$-partite $r$-graphs, loose cycles\COMMENT{We assume that $r\geq3$. Otherwise this is covered by the comment about tight cycles. Assume we perform a breath-first search. For a loose cycle, each step in the BFS reveals only one new edge at each side of the path around the first edge. If the cycle has odd length $m$, then $D_F=\lceil m/2\rceil$. If we start from an edge $e$, after $D_F-1$ steps we have already covered all the edges, and thus $V_{D_F}(e)=\varnothing$. Suppose now the cycle has even length. Then $D_F=m/2+1$ and again $V_{D_F}(e)=\varnothing$.}
and tight cycles\COMMENT{A tight cycle $C$ on $n$ vertices has diameter $D=\lceil \frac{\lfloor n/2\rfloor-1}{r-1}\rceil$ (though that number is not important for the argument). Consider any vertex $x$ on $C$ and let $P$ be the loose subpath of $C$ consisting of $D-1$ edges `to the left' and $D-1$ edges `to the right' of $x$. (So $x$ is the centre of $P$.) Then $P$ leaves at least one and at most $2(r-1)$ vertices of $C$ uncovered (since $D$ is the diameter of $C$). Now consider the edge $e$ of $C$ which contains $x$ as its leftmost vertex. Let $y$ be the rightmost vertex of $e$ and consider the loose subpath $P'$ of $C$ which consists of $e$ as well as $D-1$ edges to the left of $x$ and $D-1$ edges to the right of $y$. Then $P'$ covers either all of $C$ or it covers precisely $r-1$ vertices more than $P$, i.e. $P'$ leaves at most $r-1$ vertices of $C$ uncovered. But $V_D(e)$ consists of precisely all those uncovered vertices, and so $V_D(e)$ cannot span any edges. As a tight cycle is completely symmetric, this completes the proof.}.
If $r=2$ then $\mathcal{F}_E$ also contains hypercubes, for example.

\begin{proposition}\label{teor:upbound15}
For every\/ $\varepsilon>0$, the following holds.
Let\/ $F\in\mathcal{F}_E$ be an\/ $r$-graph and let\/ $D$ be its diameter.
For the class consisting of all input\/ $r$-graphs\/ $G$ with average degree\/ $d$ and maximum degree\/ $\Delta(G)=\bigO(d)$, there exists an $\varepsilon$-tester for\/ $F$-freeness with\/ $\bigO_\varepsilon(d^{D})$ queries.
\end{proposition}

\begin{proof}
We consider a one-sided error $\varepsilon$-tester, which works in a very similar way as in the proof of \cref{teor:upbound1}.
The $F$-freeness tester chooses a set $S\subseteq V$ of size $\Theta(1/\varepsilon)$ uniformly at random.
It then chooses an edge $e$ incident to each $v\in S$ uniformly at random and finds $G\{e,D\}$ by performing neighbour queries; then, it searches for a copy of $F$.
If any copy of $F$ is found, the algorithm rejects the input; otherwise, it accepts.
The query complexity is clearly $\bigO(d^{D}/\varepsilon)$.
The analysis of the algorithm is similar to that of \cref{teor:upbound1}, so we omit the details.\COMMENT{If $G$ is $F$-free, then the algorithm will never find a copy of $F$ and will always accept the input.
Assume now that $G$ is $\varepsilon$-far from being $F$-free.
As $G$ is $\varepsilon$-far from being $F$-free, it contains at least $\varepsilon nd$ edges that belong to copies of $F$.
By uniformly selecting a vertex in $V$ and uniformly selecting an edge in $G$ incident to it (note that this requires $\Theta(\log d)$ queries (if the algorithm has a bound on the maximum degree), but this is an additive term and it is of lower order than the polynomial in $d$), the probability that any fixed edge is chosen is $\Omega(1/nd)$.
As $G$ contains at least $\varepsilon nd$ edges that belong to copies of $F$, by repeating this $\Theta(1/\varepsilon)$ times, the probability that we find at least one edge belonging to a copy of $F$ is at least $2/3$.
By the definition of $\mathcal{F}_E$, any copy of $F$ that the edge belongs to will be contained in $G\{e,D\}$.
One can then use any procedure on $G\{e,D\}$ to reveal said copy, and then the algorithm rejects the input.}
\end{proof}

We conclude with the following bound, which works for arbitrary $G$ and any $F$ without isolated vertices.
Given an $r$-graph $F$, we define its \emph{vertex-overlap index} $\ell(F)$ as the minimum integer $\ell$ such that two graphs isomorphic to\/ $F$ sharing\/ $\ell$ vertices must share at least one edge; if this does not hold for any $\ell\in[v_F]$, we then set $\ell=v_F+1$.
For instance, $\ell(K_{k}^{(r)})=r$, and for a matching $M$ we have $\ell(M)=|V(M)|+1$ if $|V(M)|\geq2r$.

\begin{theorem}\label{teor:upbound3}
For every\/ $\varepsilon>0$, the following holds.
Let\/ $r\geq2$ and let\/ $F$ be an\/ $r$-graph without isolated vertices.
Let\/ $\ell\coloneqq\ell(F)$.
For the class consisting of all input\/ $r$-graphs\/ $G$ on\/ $n$ vertices with average degree\/ $d$ and maximum degree\/ $\Delta$, there exists an $\varepsilon$-tester for\/ $F$-freeness with\/ $\bigO_\varepsilon(\max\{(n/(nd)^{1/v_F})^{r},(n^{\ell-2}\Delta/d)^{r/(\ell-1)}\})$ queries.
\end{theorem}

In the case when $F=K_k^{(r)}$ and the input $r$-graph $G$ satisfies $\Delta(G)=\bigO(d)$, the bound in \cref{teor:upbound3} becomes $\bigO_\varepsilon((n/(nd)^{1/k})^r)$ whenever $d=o(n^{k/(r-1)-1})$, and $\bigO_\varepsilon(n^{r(r-2)/(r-1)})$ otherwise\COMMENT{Note that if $d=xn^{k/(r-1)-1}$ then
\[\frac{(n/(nd)^{1/k})^r}{nd}=\frac{(n/(x^{1/k}n^{1/(r-1)}))^r}{nd}=\frac{x^{-r/k}n^{r(r-2)/(r-1)}}{xn^{k/(r-1)}}=x^{-r/k-1}n^{r(r-2)/(r-1)-k/(r-1)}.\]
Note that the second bound beats $nd$ if $d$ is very large.
The first beats $nd$ if $x$ is not too small and $k>r(r-2)$ (note that this holds for graphs).}.

\begin{proof}
Choose a constant $c$ which is large enough compared to $v_F$ and $e_F$.
We present a one-sided error $\varepsilon$-tester in \cref{algo:tester5}.
In this proof, the constants in the $\bigO$ notation are independent of $c$\COMMENT{and $\varepsilon$ whenever it is not a subindex, but this is already the way we defined the notation in this section}.

\begin{algorithm}
\begin{algorithmic}[1]
\Procedure{Canonical $F$ tester}{}
\State{Let $s=c\max\{n/(\varepsilon nd)^{1/v_F},(n^{\ell-2}\Delta/\varepsilon d)^{1/(\ell-1)}\}$.}\label{algo5step1}
\State{Choose a set $S\subseteq V$ of size $s$ uniformly at random.}
\State{Find $G[S]$ by performing all vertex-set queries.}
\State{\textbf{if} $G[S]$ contains a copy of $F$, \textbf{then} reject.}\label{alg1reject}
\State{\textbf{otherwise}, accept.}
\EndProcedure
\end{algorithmic}
\caption{An $F$-freeness $\varepsilon$-tester for $r$-graphs.}\label{algo:tester5}
\end{algorithm}

It is easy to see that we may assume $s$ is large compared to $v_F$\COMMENT{Indeed, ignoring $c$, we have that $s$ is at least a constant, so by making $c$ large enough, $s$ will be large enough. Indeed, for $s$ to be at most a constant we need $\varepsilon=\Omega(n^{v_F-1}/d)$. But $\varepsilon\leq1$. So this is only achieved when $d=\Theta(n^{r-1})$ and $F$ is only one edge (as we assume that $F$ has no isolated vertices, we have $v_F\geq r$). But if $F$ is an edge and $G$ is dense there is a constant time tester.}.
If $G$ is $F$-free, the algorithm will never find a copy of $F$ and will always accept the input.
Assume now that $G$ is $\varepsilon$-far from being $F$-free.
Then, $G$ must contain a set $\mathcal{F}$ of $\varepsilon nd/e_F$ edge-disjoint copies of $F$\COMMENT{Since it is $\varepsilon$-far from $F$-free, we have to delete at least $\varepsilon nd$ edges to make it $F$-free. Consider a maximal family $\mathcal{F}$ of edge-disjoint copies of $F$; if we delete all the edges in all of these copies, then the $r$-graph becomes $F$-free. Therefore, $e_F|\mathcal{F}|\geq\varepsilon nd$.}.
For each $W\subseteq V$, we define $\deg_{\mathcal{F}}(W)\coloneqq|\{F'\in\mathcal{F}:W\subseteq V(F')\}|$.
It is clear that 
\begin{equation}\label{equa:Fdegbound}
\deg_{\mathcal{F}}(W)\leq\min_{v\in W}\deg(v)\leq\Delta.
\end{equation}

For any fixed $F'\in\mathcal{F}$, we have $\mathbb{P}[F'\in G[S]]=(1\pm1/2)(s/n)^{v_F}$\COMMENT{The probability is $\binom{n-v_F}{s-v_F}/\binom{n}{s}=(s)_{v_F}/(n)_{v_F}$.}.
We denote by $X$ the number of $F'\in\mathcal{F}$ such that $F'\in G[S]$.
We conclude that
\begin{equation}\label{equa:up3exp}
\mathbb{E}[X]=(1\pm 1/2)|\mathcal{F}|\left(\frac{s}{n}\right)^{v_F}=\Theta\left(\frac{\varepsilon ds^{v_F}}{n^{v_F-1}}\right).
\end{equation}
The variance of $X$ can be estimated by observing that we only need to consider $r$-graphs $F',F''\in\mathcal{F}$ whose vertex sets intersect, as otherwise the events are negatively correlated\COMMENT{Indeed, for each vertex-disjoint pair $F',F''\in\mathcal{F}$ we have that $\mathbb{P}[F'\in G[S]]=\mathbb{P}[F''\in G[S]]=\binom{n-v_F}{s-v_F}/\binom{n}{s}=(s)_{v_F}/(n)_{v_F}$, while $\mathbb{P}[F'\cup F''\in G[S]]=\binom{n-2v_F}{s-2v_F}/\binom{n}{s}=(s)_{2v_F}/(n)_{2v_F}<((s)_{v_F})^2/((n)_{v_F})^2=\mathbb{P}[F'\in G[S]]^2$.}.
Hence,\COMMENT{See previous note for first inequality:\begin{align*}\mathrm{Var}[X]=&\sum_{(F',F'')\in\mathcal{F}\times\mathcal{F}}\mathrm{Cov}[\mathbf{1}_{F'\subseteq G[S]},\mathbf{1}_{F''\subseteq G[S]}]=\sum_{(F',F'')\in\mathcal{F}\times\mathcal{F}}(\mathbb{E}[\mathbf{1}_{F'\subseteq G[S]}\mathbf{1}_{F''\subseteq G[S]}]-\mathbb{E}[\mathbf{1}_{F'\subseteq G[S]}]\mathbb{E}[\mathbf{1}_{F''\subseteq G[S]}])\\=&\sum_{(F',F'')\in\mathcal{F}\times\mathcal{F}}(\mathbb{P}[F'\cup F''\subseteq G[S]]-\mathbb{P}[F'\subseteq G[S]]\mathbb{P}[F''\subseteq G[S]])\\=&\sum_{\substack{(F',F'')\in\mathcal{F}\times\mathcal{F}\\V(F')\cap V(F'')=\varnothing}}(\mathbb{P}[F'\cup F''\subseteq G[S]]-\mathbb{P}[F'\subseteq G[S]]\mathbb{P}[F''\subseteq G[S]])\\+&\sum_{\substack{(F',F'')\in\mathcal{F}\times\mathcal{F}\\V(F')\cap V(F'')\neq\varnothing}}(\mathbb{P}[F'\cup F''\subseteq G[S]]-\mathbb{P}[F'\subseteq G[S]]\mathbb{P}[F''\subseteq G[S]])\\\leq&\sum_{\substack{(F',F'')\in\mathcal{F}\times\mathcal{F}\\V(F')\cap V(F'')\neq\varnothing}}(\mathbb{P}[F'\cup F''\subseteq G[S]]-\mathbb{P}[F'\subseteq G[S]]\mathbb{P}[F''\subseteq G[S]])\leq\sum_{\substack{(F',F'')\in\mathcal{F}\times\mathcal{F}\\ V(F')\cap V(F'')\neq\varnothing}}\mathbb{P}[F'\cup F''\subseteq G[S]]\end{align*}}
\begin{equation}\label{equa:up3var}
\mathrm{Var}[X]\leq\sum_{\substack{(F',F'')\in\mathcal{F}\times\mathcal{F}\\ V(F')\cap V(F'')\neq\varnothing}}\mathbb{P}[F'\cup F''\subseteq G[S]]
=\sum_{i=1}^{v_F}\sum_{\substack{(F',F'')\in\mathcal{F}\times\mathcal{F}\\ |V(F')\cap V(F'')|=i}}\mathbb{P}[F'\cup F''\subseteq G[S]].
\end{equation}
Let us estimate this quantity for each $i\in[v_F]$.
For $i\in[v_F-1]$ we can apply a double counting argument to see that 
\begin{equation}\label{equa:up3eq1}
|\{(F',F'')\in\mathcal{F}\times\mathcal{F}:|V(F')\cap V(F'')|=i\}|\leq2\sum_{W\in\binom{V}{i}}\binom{\deg_{\mathcal{F}}(W)}{2},
\end{equation}
while for $i=v_F$ we have that
\begin{equation}\label{equa:up3eq2}
|\{(F',F'')\in\mathcal{F}\times\mathcal{F}:|V(F')\cap V(F'')|=v_F\}|\leq|\mathcal{F}|\COMMENT{As we consider ordered pairs, these are always here, even if $\ell\leq v_F$.}+2\sum_{W\in\binom{V}{v_F}}\binom{\deg_{\mathcal{F}}(W)}{2}.
\end{equation}
Note that\COMMENT{as each $F'\in\mathcal{F}$ is counted each time $W\subseteq V(F')$, which happens exactly $\binom{v_F}{i}$ times, \[\sum_{W\in\binom{V}{i}}\deg_{\mathcal{F}}(W)=\binom{v_F}{i}|\mathcal{F}|.\]}
\begin{equation}\label{equa:up3eq25}
\sum_{W\in\binom{V}{i}}\deg_{\mathcal{F}}(W)=\bigO(|\mathcal{F}|)=\bigO(\varepsilon nd).
\end{equation}
By assumption on $F$, we have that $\deg_{\mathcal{F}}(W)\leq1$ for all $W$ such that $|W|\geq\ell$, which implies that $\binom{\deg_{\mathcal{F}}(W)}{2}=0$.
Moreover, by \eqref{equa:Fdegbound} and \eqref{equa:up3eq25}, for each $i\in[\ell-1]$ we obtain
\begin{equation}\label{equa:up3eq3}
\sum_{W\in\binom{V}{i}}\binom{\deg_{\mathcal{F}}(W)}{2}\leq\Delta\sum_{W\in\binom{V}{i}}\deg_{\mathcal{F}}(W)=\bigO(\varepsilon nd\Delta).
\end{equation}
Combining \eqref{equa:up3eq1}--\eqref{equa:up3eq3}, the estimation in \eqref{equa:up3var} yields
\begin{equation}\label{equa:up3var2}
\mathrm{Var}[X]=\bigO\left(\varepsilon nd\left(\frac{s}{n}\right)^{v_F}\right)+\sum_{i=1}^{\ell-1}\bigO\left(\varepsilon nd\Delta\left(\frac{s}{n}\right)^{2v_F-i}\right)=\varepsilon nd\cdot\bigO\left(\left(\frac{s}{n}\right)^{v_F}+\Delta\left(\frac{s}{n}\right)^{2v_F-\ell+1}\right).
\end{equation}

By Chebyshev's inequality, $\mathbb{P}[X=0]\leq\mathrm{Var}[X]/\mathbb{E}[X]^2$.
Using \eqref{equa:up3exp}, \eqref{equa:up3var2} and the fact that $c$ is large compared to $v_F$ and $e_F$, one can check that $\mathrm{Var}[X]/\mathbb{E}[X]^2<1/3$\COMMENT{We claim that $\mathbb{P}[X=0]\leq C(c^{-v_F}+c^{-\ell+1})$, where $C$ is some constant, so the conclusion follows by taking $c$ large enough (and observing that $\ell\geq r\geq 2$).\newline
To prove so, consider the two possible cases, $s=cn/(\varepsilon nd)^{1/v_F}$ (1) and $s=c(\Delta n^{\ell-1}/(\varepsilon nd))^{1/(\ell-1)}$ (2).
Observe that (1) holds if
\[n/(\varepsilon nd)^{1/v_F}\geq(\Delta n^{\ell-1}/(\varepsilon nd))^{1/(\ell-1)}\iff\Delta^{1/(\ell-1)}(\varepsilon nd)^{1/v_F-1/(\ell-1)}\leq1\iff\Delta(\varepsilon nd)^{(\ell-1)/v_F-1}\leq1\]
and (2) holds if
\[n/(\varepsilon nd)^{1/v_F}\leq(\Delta n^{\ell-1}/(\varepsilon nd))^{1/(\ell-1)}\iff\Delta^{-1/(\ell-1)}(\varepsilon nd)^{1/(\ell-1)-1/v_F}\leq1\iff(\varepsilon nd/\Delta)^{v_F/(\ell-1)}/(\varepsilon nd)\leq1.\]
By \eqref{equa:up3exp}, there exists a constant $C_e$ such that $\mathbb{E}[X]^2\geq C_e(\frac{\varepsilon nds^{v_F}}{n^{v_F}})^2$.
Similarly, by \eqref{equa:up3var2}, there is a constant $C_v$ such that $\mathrm{Var}[X]\leq C_v\varepsilon nd(\left(\frac{s}{n}\right)^{v_F}+\Delta\left(\frac{s}{n}\right)^{2v_F-\ell+1})$.
Combining these, in the first case we have that $s/n=c/(\varepsilon nd)^{1/v_F}$ and thus
\[\mathbb{P}[X=0]\leq\frac{C_v}{C_e}\frac{1}{\varepsilon nd}\left(\left(\frac{s}{n}\right)^{-v_F}+\Delta\left(\frac{s}{n}\right)^{-\ell+1}\right)=\frac{C_v}{C_e}\left(\frac{1}{c^{v_F}}+\frac{\Delta(\varepsilon nd)^{(\ell-1)/v_F-1}}{c^{\ell-1}}\right)\leq\frac{C_v}{C_e}(c^{-v_F}+c^{-\ell+1}),\]
where the last inequality holds by the first observation.
In the second, $s/n=c(\Delta/(\varepsilon nd))^{1/(\ell-1)}$ and by the second observation
\[\mathbb{P}[X=0]\leq\frac{C_v}{C_e}\frac{1}{\varepsilon nd}\left(\left(\frac{s}{n}\right)^{-v_F}+\Delta\left(\frac{s}{n}\right)^{-\ell+1}\right)=\frac{C_v}{C_e}\left(\frac{(\varepsilon nd/\Delta)^{v_F/(\ell-1)}/(\varepsilon nd)}{c^{v_F}}+\frac{1}{c^{\ell-1}}\right)\leq\frac{C_v}{C_e}(c^{-v_F}+c^{-\ell+1}).\]
It is easy to check that $C_v$ and $C_e$ depend only on $v_F$ and $e_F$, which explains why we need $c$ large with respect to these paramenters.
}.
Thus $G[S]$ contains a copy of $F$ with probability at least $2/3$.
Therefore, $G$ will be rejected with probability at least $2/3$, which shows that \cref{algo:tester5} is an $F$-freeness $\varepsilon$-tester.

The query complexity of the algorithm is given by performing all $\binom{s}{r}$ vertex-set queries.
This yields the stated complexity.
\end{proof}

\section*{Acknowledgements}

We would like to thank the anonymous referee for a careful reading and the valuable comments provided.




\appendix

\section{Proof of \cref{lema:disjointcopies}}\label{appendix1}

\begin{proof}[Proof of \cref{lema:disjointcopies}]
Clearly, $D_K\geq D_F$ for each $K\subseteq F$, and $D_K\leq X_K$.
Thus, $D_F\leq\min\{X_K:K\subseteq F, e_K>0\}$.
\Cref{coro:copies} applied to each $K$ implies that $\min\{X_K:K\subseteq F, e_K>0\}=\Theta(\Phi_{F})$ a.a.s., so $D_F=O(\Phi_{F})$ a.a.s.

It now suffices to show that $D_F=\Omega(\Phi_{F})$ a.a.s.
To do so, for each $G\in\mathcal{G}_{n,d}^{(r)}$ we define an auxiliary graph $\Gamma=\Gamma(G)$ whose vertices are all the copies of $F$ in $G$, and in which $F_1,F_2\in V(\Gamma)$ are adjacent if and only if $F_1$ and $F_2$ share at least one edge.
Let us denote by $\sum\nolimits^*$ the sum over all graphs $\tilde{F}$ which can be written as $\tilde{F}=F'\cup F''$, where $F', F''\subseteq K_V^{(r)}$, $F', F''\cong F$ and $E(F')\cap E(F'')\neq\varnothing$.
This means that $v_{\Gamma}=X_F$ and \[e_{\Gamma}=\bigO\left(\sum\nolimits^*X_{\tilde{F}}\right).\]
Note that the size of the largest independent set in $\Gamma$ equals $D_F$.
By Tur\'an's theorem\COMMENT{Tur\'an's theorem states that a $K_{r+1}$-free graph $G$ on $n$ vertices has at most $(r-1)n^2/2r$ edges. Let $\overline{r}$ be the size of the biggest independent set in $G$ so that $\overline{G}$ is $K_{\overline{r}+1}$-free. Then, \[\binom{n}{2}-e(G)=e(\overline{G})\leq(\overline{r}-1)n^2/2\overline{r}=\frac{1}{2}\left(1-\frac{1}{\overline{r}}\right)n^2.\] By rearranging this one obtains $\overline{r}\geq n^2/(2e(G)+n)$.} we have that
\[D_F\geq\frac{X_F^2}{X_F+\bigO\left(\sum\nolimits^*X_{\tilde{F}}\right)}.\]
Using \cref{coro:copies}, one can check that, a.a.s.\COMMENT{$\frac{X_F^2}{X_F+2\sum\nolimits^*X_{\tilde{F}}}=\Omega(\Phi_{F})\Longleftrightarrow X_F+2\sum\nolimits^*X_{\tilde{F}}=\frac{X_F^2}{\Omega(\Phi_{F})}=\bigO(X_F^2/\Phi_{F})$.
By \cref{coro:copies}, this a.a.s.~equals $\bigO(n^{2v_F}p^{2e_F}\Phi_{F}^{-1})$.
By definition of $\Phi_{F}$ and \cref{coro:copies}, $X_F=\bigO(X_F^2\Phi_{F}^{-1})$.},
\begin{equation}\label{equa:disj}
\frac{X_F^2}{X_F+\bigO\left(\sum\nolimits^*X_{\tilde{F}}\right)}=\Omega(\Phi_{F})\Longleftrightarrow\,X_{\tilde{F}}=\bigO\left({n^{2v_F}p^{2e_F}}{\Phi_{F}^{-1}}\right)
\end{equation}
for all $\tilde{F}=F'\cup F''$ such that $e_{F'\cap F''}>0$\COMMENT{as there is only a constant number of such $r$-graphs $\tilde{F}$}.
So it suffices to prove the final bound in \eqref{equa:disj}.

For any fixed $r$-graph $K$, let $\Psi_K\coloneqq n^{v_K}p^{e_K}$.
Note that if $K\subseteq F$, then $\Psi_K=\Theta(\mathbb{E}[X_K])$ (by \cref{coro:expectation}\ref{item3}).
Furthermore, for any two $r$-graphs $K$ and $L$ on a vertex set $V$,
\begin{equation}\label{equa:psimodularity}
\Psi_L\Psi_K=\Psi_{L\cup K}\Psi_{L\cap K}.
\end{equation}

Consider two copies $F'$ and $F''$ of $F$ whose intersection has at least one edge.
Let $\tilde{F}\coloneqq F'\cup F''$ and $K\coloneqq F'\cap F''$, so $e_K>0$.
Thus, by \cref{coro:expectation}\ref{item3},
\begin{equation}\label{equa:disjexpectation}
\mathbb{E}[X_{\tilde{F}}]=\Theta\left(n^{2v_F-v_K}p^{2e_F-e_K}\right)=\Theta\left(\frac{\Psi_F^2}{\Psi_K}\right)=O\left(\frac{\Psi_F^2}{\Phi_{F}}\right).
\end{equation}

We now claim that
\begin{equation}\label{equa:boundontildePhi}
\Phi_{\tilde{F}}=\min\{\mathbb{E}[X_L]:L\subseteq\tilde{F}, e_L>0\}=\Omega\left(\min\left\{\frac{\Phi_{F}^3}{\Psi_K^2},\frac{n\Phi_{F}^2}{\Psi_K^2}\right\}\right).
\end{equation}
Indeed, consider any $r$-graph $L\subseteq{\tilde{F}}$ with $e_L>0$ and let $L'\coloneqq L\cap F'$ and $L''\coloneqq L\cap F''$.
Note that $L\cup K=(L'\cup K)\cup (L''\cup K)$ and $K=(L'\cup K)\cap (L''\cup K)$, so two applications of \eqref{equa:psimodularity} yield
\begin{equation}\label{equa:psiapp}
\Psi_L=\frac{\Psi_{L\cup K}\Psi_{L\cap K}}{\Psi_K}=\frac{\Psi_{L'\cup K}\Psi_{L''\cup K}\Psi_{L\cap K}}{\Psi_K^2}.
\end{equation}
If $e_{L\cap K}>0$, then the values of $\Psi_{L'\cup K}$, $\Psi_{L''\cup K}$ and $\Psi_{L\cap K}$ can be lower bounded by $\Omega(\Phi_{F})$ (as $L'\cup K$, $L''\cup K$ and $L\cap K$ are all subgraphs of $F$).
Thus $\Psi_L=\Omega(\Phi_{F}^3/\Psi_K^2)$.
So suppose that $e_{L\cap K}=0$.
Then $L'$ and $L''$ are edge-disjoint, and at least one of them has at least one edge.
We may assume that $e_{L'}>0$ without loss of generality.
Consider three cases.
If $e_{L''}=0$, then $\Psi_{L''}=n^{v_{L''}}\geq\Psi_{L'\cap L''}$ and, by \eqref{equa:psimodularity}, $\Psi_L=\Psi_{L'}\Psi_{L''}/\Psi_{L'\cap L''}=\Omega(\Psi_{L'})=\Omega(\Phi_{F})=\Omega(\Phi_{F}^3/\Psi_K^2)$, where the final equality holds since $\Psi_K=\Omega(\Phi_{F})$.
If $e_{L''}>0$ but $L'\cap L''=\varnothing$ then $\Psi_L=\Psi_{L'}\Psi_{L''}=\Omega(\Psi_{L'})\COMMENT{As $\Psi_{L''}=\Omega(\Phi_{F})=\omega(1)$}=\Omega(\Phi_{F}^3/\Psi_K^2)$.
Otherwise, we have that $e_{L''}>0$ and $L'\cap L''\neq\varnothing$.
We use \eqref{equa:psiapp} taking into account that $\Psi_{L\cap K}=n^{v_{L\cap K}}=\Omega(n)$ to conclude that $\Psi_L=\Omega(n\Phi_{F}^2/\Psi_K^2)$.
This proves the claim.

By \cref{lema:variance} we have that $\mathrm{Var}(X_{\tilde{F}})=\mathbb{E}[X_{\tilde{F}}]^2\bigO(\varepsilon_{n,d}+\Phi_{\tilde{F}}^{-1})$.
As $\varepsilon_{n,d}=o(1)$ by assumption, by \eqref{equa:disjexpectation} Chebyshev's inequality implies that the final bound in \eqref{equa:disj} holds a.a.s.~if $\Phi_{\tilde{F}}^{-1}=\bigO(\varepsilon_{n,d})$.
Therefore, we may assume that $\mathrm{Var}(X_{\tilde{F}})=\bigO({\mathbb{E}[X_{\tilde{F}}]^2}/{\Phi_{\tilde{F}}})=\bigO({\Psi_{\tilde{F}}^2}/{\Phi_{\tilde{F}}})$.
Consequently, by \eqref{equa:boundontildePhi} and \eqref{equa:psimodularity} we have 
\[\mathrm{Var}(X_{\tilde{F}})=\bigO\left(\frac{\Psi_{\tilde{F}}^2\Psi_K^2}{\Phi_{F}^3}+\frac{\Psi_{\tilde{F}}^2\Psi_K^2}{n\Phi_{F}^2}\right)=\bigO\left(\frac{\Psi_F^4}{\Phi_{F}^3}+\frac{\Psi_F^4}{n\Phi_{F}^2}\right).\]
Thus, Chebyshev's inequality gives
\[\mathbb{P}\left[X_{\tilde{F}}\geq\mathbb{E}[X_{\tilde{F}}]+\frac{\Psi_F^2}{\Phi_{F}}\right]=\bigO\left(\Phi_{F}^{-1}+1/n\right)=o(1)\]
by assumption.
Hence $X_{\tilde{F}}=O\left(\Psi_F^2/\Phi_{F}\right)$ a.a.s., as required.
\end{proof}

\end{document}